\documentclass[preprint,11pt]{elsarticle}

\usepackage{amsmath,amsfonts,amssymb,graphicx,listings,array,amsthm}
\usepackage{wrapfig, booktabs, xcolor, subcaption, chngcntr}
\usepackage{multirow, geometry, bm, commath, hyperref}
\usepackage{algorithm,algpseudocode,float}

\algnewcommand\algorithmicinput{\textbf{Input:}}
\algnewcommand\INPUT{\item[\algorithmicinput]}
\algnewcommand\algorithmicoutput{\textbf{Output:}}
\algnewcommand\OUTPUT{\item[\algorithmicoutput]}

\makeatletter
\newenvironment{breakablealgorithm}
  {% \begin{breakablealgorithm}
   \begin{center}
     \refstepcounter{algorithm}% New algorithm
     \hrule height.8pt depth0pt \kern2pt% \@fs@pre for \@fs@ruled
     \renewcommand{\caption}[2][\relax]{% Make a new \caption
       {\raggedright\textbf{\ALG@name~\thealgorithm} ##2\par}%
       \ifx\relax##1\relax % #1 is \relax
         \addcontentsline{loa}{algorithm}{\protect\numberline{\thealgorithm}##2}%
       \else % #1 is not \relax
         \addcontentsline{loa}{algorithm}{\protect\numberline{\thealgorithm}##1}%
       \fi
       \kern2pt\hrule\kern2pt
     }
  }{% \end{breakablealgorithm}
     \kern2pt\hrule\relax% \@fs@post for \@fs@ruled
   \end{center}
  }
\makeatother

\usepackage{color,soul}
\makeatletter
\def\ps@pprintTitle{%
  \let\@oddhead\@empty
  \let\@evenhead\@empty
  \def\@oddfoot{\reset@font\hfil\thepage\hfil}
  \let\@evenfoot\@oddfoot
}
\makeatother

\DeclareCaptionFormat{myformat}{\fontsize{9}{10}\selectfont#1#2#3}
\numberwithin{equation}{section}
\newcounter{rowno}
\newcommand{\phikthree}{\varphi_{k_3}}
\newcommand{\phinull}{\varphi_{0}}
\newcommand{\fkthree}{f_{k_3}}
\newcommand{\fknull}{f_{0}}
\newcommand{\bessel}{{\bf K_0}}
\newcommand{\mi}{\texttt{m}}
\renewcommand{\ni}{\texttt{n}}

\newcommand{\subR}{{_R}}
\newcommand{\wt}{\widetilde}
\newcommand{\wh}{\widehat}
\newcommand{\bx}{\textbf{x}}

\newcommand{\bk}{\textbf{k}}
\newcommand{\bka}{\textbf{$\boldsymbol\kappa$}}
\newcommand{\br}{\textbf{r}}
\newcommand{\bp}{\textbf{p}}
\newcommand{\bn}{\textbf{n}}
\newcommand{\bs}{\textbf{s}}
\newcommand{\erfc}{\text{erfc}}

\newcommand{\ii}{\textrm{i}}
\newcommand{\im}{\mathsf{Im}}

\newcommand{\res}{\mathsf{Res}}
\newcommand{\C}{\textsf{C}}
\newcommand{\red}{\color{red}}
\newcommand{\blue}{\color{blue}}

\newcommand{\order}[1]{{\cal O}(#1)}

\newcommand{\n}{\mathtt{n}}
\newcommand{\rtwo}{\mathbb{R}^2}
\newcommand{\e}{{\mathcal{E}}}

\newcommand{\E}{\text{E$_1$}}
\newcommand{\F}{\mathrm{F}}
\newcommand{\R}{\mathrm{R}}
\renewcommand{\P}{\mathrm{P}}
\renewcommand{\c}{\mathrm{c}}
\newcommand{\nl}{n_{\mathrm{l}}}
\renewcommand{\sl}{s_{\mathrm{l}}}

\def\reals{{{\rm l} \kern-.15em {\rm R}}}
\newcommand{\subindex}[2]{{\begin{subarray}{l} {#1} \\ {#2}\end{subarray}}}
\newtheorem{theorem}{Theorem}
\theoremstyle{definition}
\newtheorem{definition}{Definition}
\newtheorem{remark}{Remark}
\newtheorem{example}{Example}

\newcommand{\etal}{\textit{et al.\ }}

\usepackage{etoolbox}% http://ctan.org/pkg/etoolbox
\AtBeginEnvironment{align}{\setcounter{subeqn}{0}}% Reset subequation number at start of align
\newcounter{subeqn} \renewcommand{\thesubeqn}{\theequation\alph{subeqn}}%
\newcommand{\subeqn}{%
  \refstepcounter{subeqn}% Step subequation number
  \tag{\thesubeqn}% Label equation
}
\begin{document}
\begin{frontmatter}
\title{The Spectral Ewald method for singly periodic domains}

\author{Davoud Saffar Shamshirgar} 
\ead{davoudss@kth.se}

\author{Anna-Karin Tornberg}
\ead{akto@kth.se}

\address{KTH Mathematics, Swedish e-Science
  Research Centre, 100 44
  Stockholm, Sweden.}


%\maketitle

\begin{abstract}
We present a fast and spectrally accurate method for efficient computation of the three dimensional Coulomb potential with periodicity in one direction. The algorithm is FFT-based and uses the so-called Ewald decomposition, which is naturally most efficient for the triply periodic case. In this paper, we show how to extend the triply periodic Spectral Ewald method to the singly periodic case, such that the cost of computing the singly periodic potential is only marginally larger than the cost of computing the potential for the corresponding triply periodic system.

In the Fourier space contribution of the Ewald decomposition, a Fourier series is obtained in the periodic direction with a Fourier integral over the non periodic directions for each discrete wave number. We show that upsampling to resolve the integral is only needed for modes with small wave numbers. For the zero wave number, this Fourier integral has a singularity. For this mode, we effectively need to solve a free-space Poisson equation in two dimensions. A very recent idea by Vico \etal makes it possible to use FFTs to solve this problem, allowing us to unify the treatment of all modes. An adaptive 3D FFT can be established to apply different upsampling rates locally. The computational cost for other parts of the algorithm is essentially unchanged as compared to the triply periodic case, in total yielding only a small increase in both computational cost and memory usage for this singly periodic case.
\end{abstract}

\end{frontmatter}

%===================================================================================================================
%===================================================================================================================

%===================================================================================================================
%===================================================================================================================
\section{Introduction}
\label{sec:intro}
In molecular dynamics simulations, a crucial and time consuming task is to compute the long-range interactions or particularly the electrostatic potential (or force) between charged particles. For systems that are subject to periodic boundary conditions, \textit{Ewald summation} is a technique to evaluate these interactions. For $N$ particles with charges $q_{\ni}$ at positions $\bx_{\ni}\in\mathsf{X}_{i=1}^3[0,L_i)$, $\ni=1,\ldots,N$, the electrostatic potential or briefly the potential evaluated at a target point $\bx_{\mi}$ is written as
\begin{align}
\varphi(\bx_{\mi}) = \sum_{\bp\in P_D}^\prime\sum_{\ni=1}^N\dfrac{q_{\ni}}{|\bx_{\mi}-\bx_{\ni}+\bp|},
\label{eq:coulomb}
\end{align}
where $P_D$ with $D=1,2,3$, can be modified such that it expresses the periodicity. The prime denotes that the term with $\ni=\mi$ is omitted from the sum for $\bp=\boldsymbol{0}$. To impose the periodicity in three dimensions (3d-periodicity), the simulation box is replicated in three directions and we define $P_3=\lbrace(\alpha_1L_1,\alpha_2L_2,\alpha_3L_3):\alpha_i\in\mathbb{Z}\rbrace$. In this case, in light of the neutrality condition, i.e., $\sum_{\ni=1}^Nq_{\ni}=0$, the sum in \eqref{eq:coulomb} is only conditionally convergent and therefore the order of summation has to be exactly defined. In the classical Ewald sum proposed by Ewald \cite{Ewald1921}, \eqref{eq:coulomb} is decomposed into a fast decaying part and a smooth part which is computed in Fourier space. The result is that of a spherical summation order, and the sum can be written as
\begin{align}
\varphi^{3\P}(\bx_{\mi}) =& \varphi^\R_{\mi}+\varphi^\F_{\mi}+\varphi^{\textrm{self}}_{\mi} \nonumber \\
=&\sum_{\bp\in P_3}^{\prime}\sum_{\ni=1}^Nq_{\ni}\frac{\erfc(\xi |\bx_{\mi}-\bx_{\ni}+\bp|)}{|\bx_{\mi}-\bx_{\ni}+\bp|} \nonumber \\
&+\frac{4\pi}{V}\sum_{\bk\neq0}\frac{e^{-k^2/4\xi^2}}{k^2}\sum_{\ni=1}^Nq_{\ni}e^{\ii\bk\cdot(\bx_{\mi}-\bx_{\ni})}-\frac{2\xi}{\sqrt{\pi}}q_{\mi},
\label{eq:ewald3p}
\end{align}
where $\mi=1,\ldots,N$, $\bk=2\pi(\frac{n_1}{L_1},\frac{n_2}{L_2},\frac{n_3}{L_3}),$ with $n_i\in\mathbb{Z}$ and $k=|\bk|$. The superscripts $R$, $F$, and $self$ denote the real-space, Fourier-space (here referred to as $k$-space), and self correction term, respectively. The self correction term is added in order to eliminate the self interaction contribution of the charges included due to the decomposition. Moreover, applying the spherical order of summation, the $\bk=0$ term (dipole term) of the 3d periodic Ewald sum, depends on the dielectric constant of the surrounding medium. If the medium has an infinite dielectric constant, the dipole term vanishes, cf. {\cite{Frenkel2002}. Assuming so, the $\bk=0$ terms is excluded from the sum.}

In \eqref{eq:ewald3p}, $\xi>0$ is the decomposition parameter (\textit{Ewald parameter}) which controls the decay of the terms in the real space and $k$-space sums, but does not alter the total result. Under the assumption of a uniform distribution of charges, a proper choice of $\xi$ can decrease the computational complexity of computing the potential \eqref{eq:ewald3p} at $\bx_{\mi}$, $\mi=1,\ldots,N$, from $\order{N^2}$ to $\order{N^{3/2}}$ albeit with a very large constant. This however can be reduced to $\order{N\log (N)}$ (also with a much smaller constant) using fast methods which take advantage of the Fast Fourier transform (FFT). Inspired by the Particle-Particle-Particle Mesh Ewald (P$^3$M) method by Hockney and Eastwood \cite{Hockney2010}, different fast methods have been proposed, including the Smooth Particle mesh Ewald (SPME) method \cite{Essmann1995}, and a spectrally accurate Ewald method for triply periodic (SE3P) \cite{Lindbo2011} and doubly periodic (SE2P) \cite{Lindbo2012} systems. Using the idea in \cite{Vico2016} we have developed another spectrally accurate method for the fast evaluation of sums involving free-space Green's functions \cite{Klinteberg2016}.

In the present work, we complete the framework of the spectral Ewald methods by extending the algorithm to systems with one periodic direction in three dimensions (1d- or singly periodic). This method can be used, e.g.,\ for simulation of nanopores \cite{Bourg2012} and nanotubes \cite{Brodka2006}. In one of the first attempts, Lekner summation was used to compute the long-range interactions in 1d-periodic system of particles \cite{Brodka2002,Arnold2005,Brodka2004}. The first derivation of the 1d-periodic Ewald summation Ewald1P was given by Porto \cite{Porto2000} with an integral representation of the $k$-space sum. The integral can however be evaluated to obtain the closed form of the formula, see \cite[Appendix D]{Tornberg2014} and references therein. Recently, Nestler \etal \cite{Nestler2015} developed a fast algorithm which employs non-equispaced FFTs (NFFT). To the best of our knowledge, this is the only method with $\order{N\log(N)}$ complexity for 1d-periodic problems. The approach that we have taken differs significantly from theirs, as will be commented on in Section \ref{sec:se1p}.

FFT based methods such as methods in the PME (particle mesh Ewald) family are most efficient for the triply periodic case. In this case, FFTs can be used in all directions without any oversampling. As soon as there is a non-periodic direction, the grid has to be extended in that direction. In the doubly periodic case, Arnold \etal \cite{Arnold2002c,DeJoannis2002d} devised a method where the problem is extended to full periodicity, with a larger length in the non-periodic direction, and where a correction term is applied to improve on the result. Here, the increased length in the non-periodic direction simply means a zero-padding of the FFT, increasing in the number of grid points in that direction. The SE2P method by Lindbo and Tornberg \cite{Lindbo2012} takes a different approach, which needs a ``mixed" transform; a discrete Fourier transform in the periodic variables and an approximation to the continuous Fourier integral transform in the free dimension. Also in this case the grid in the free dimension must be oversampled for an accurate approximation. 

Extending the Spectral Ewald method to the singly periodic case, there were two main challenges to overcome. Firstly, the oversampling. An oversampling by a factor of four to six makes the FFTs four to six times more expensive to compute when the oversampling is applied in one dimension. With two free dimensions, this would increase the cost by a factor of 16 to 64, which is clearly not desirable.  Oversampling needs to be done to resolve the Fourier integrals. By analysis of a similar one-dimensional integral we can understand how the error behaves and recognize that only for small discrete wave numbers (a small number of periodic modes) do the FFT grids need to be upsampled. Based on this, we have developed what we call adaptive FFTs and IFFTs (denoted by AFT and AIFT in this paper) that only upsample for a select number of discrete modes in the periodic direction. The ratio of the run time for the AFT and the FFT without oversampling decreases with grid size, as a smaller fraction of modes must be oversampled, and a typical increase in cost can be a factor of 2-3 instead.

In the derivation of the singly periodic Ewald sums, there is a term that includes the contribution from the zero wave number in the periodic direction, i.e., that depends only on the variables in the free directions. The direct evaluation of this sum at all target points would however incur an $\order{N^2}$ computational cost, and the second challenge was to significantly reduce this. One interpretation of this sum is that it is the solution to the Poisson equation in $\mathbb{R}^2$, with the right hand side convolved with specifically scaled Gaussians centered at each of the charge locations and projected onto the plane $z=0$, with $z$ the periodic direction. A very recent idea for how to solve free space problems by the means of FFTs \cite{Vico2016} can therefore be used. Hence, we are able to include this zero wave number contribution into the treatment of the full $k$-space term. This extra sum then only entails a special scaling in Fourier space for the modes with a zero wave number in the periodic direction. This is done at a negligible extra cost.

With these two main advances, we have developed a fast and spectrally accurate FFT-based method for the evaluation of the $k$-space sum in the Ewald summation formula for the singly periodic case. We will denote this method the SE1P method. 

The outline of this paper is as follows; In section 2, we present the Ewald1P formulas and provide the truncation error estimates. We introduce our fast method  for computing the electrostatic force and potential in section 3. Section 4 is devoted to a discussion regarding approximation errors including the errors introduced by discretization of the Fourier integrals, and the related issue of parameter selection including oversampling rates in the adaptive FFT. The following section is dedicated to the implementation details of the algorithm. Finally, we supply numerical results including comparisons with the 3d-periodic counterpart in section 6 and wrap up with conclusions in section 7.

%===================================================================================================================
%===================================================================================================================

%===================================================================================================================
%===================================================================================================================
\section{Singly Periodic Ewald summation, Ewald1P}
\subsection{The Ewald summation formula}
\label{sec:ewald1p}
In this section, we present the three dimensional Ewald summation formula under 1d-periodic boundary conditions. The formulas to compute \eqref{eq:coulomb} can be derived using Fourier integrals. The reader may consult \cite{Tornberg2014} for details and alternatives of the derivation. We note here that unlike the 3d-periodic case, equation \eqref{eq:coulomb} is shown to be absolutely convergent in the 1d-periodic case under the assumption of charge neutrality and is therefore independent of the summation order \cite{Arnold2005}.

\begin{figure}[tp]
\centering \includegraphics[width=0.4\linewidth]{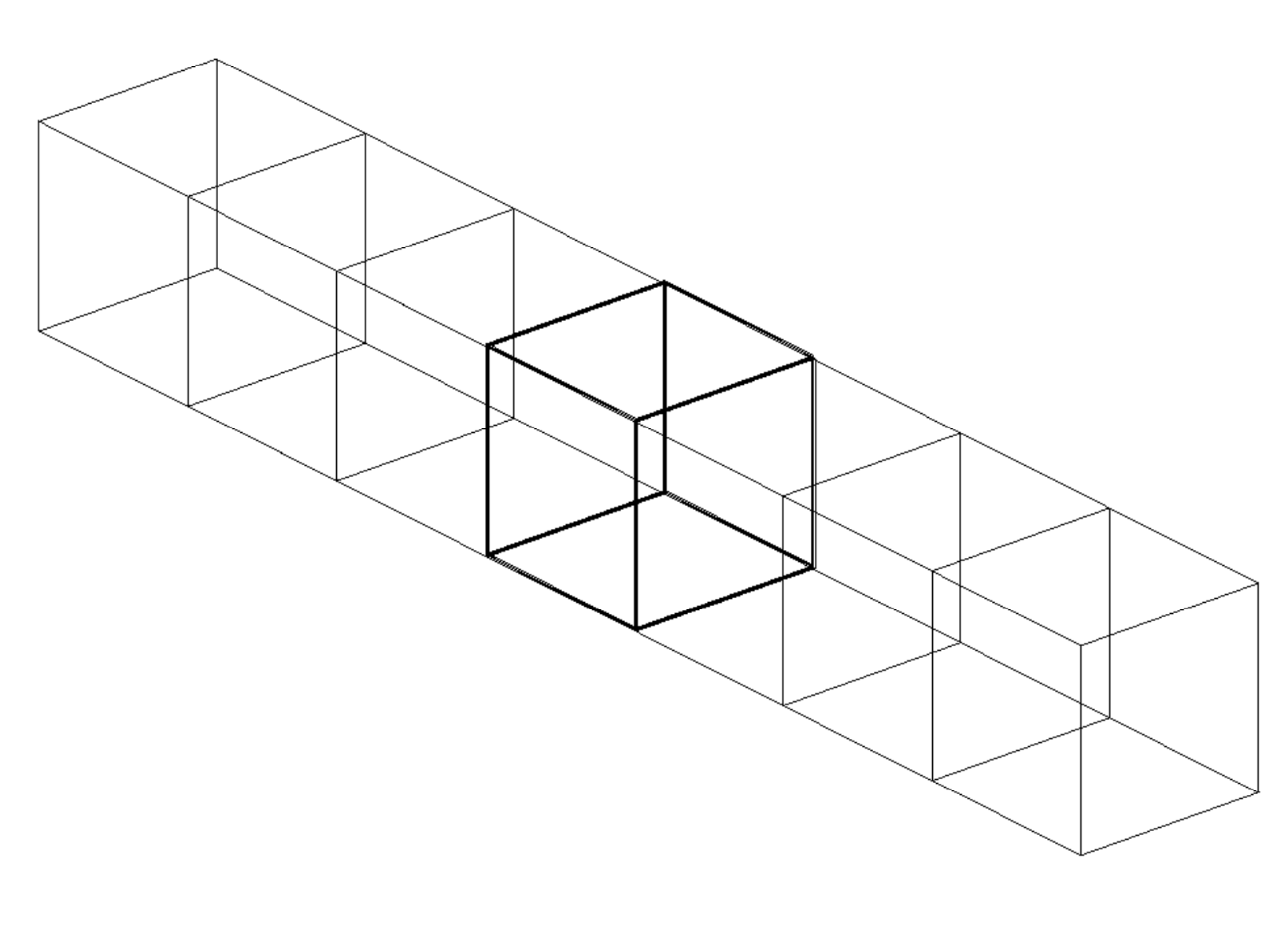}
\caption{In 1d-periodic systems, the original simulation box is replicated in 1 direction.}
\label{fig:1pcube}
\end{figure}

Henceforth we assume that the third dimension is periodic and the other two are free. The single periodicity is imposed by setting $\bp\in P_1=\lbrace(0,0,\alpha_3L_3):\alpha_3\in\mathbb{Z}\rbrace$ in \eqref{eq:coulomb}, see figure \ref{fig:1pcube}. Let $\bx=(\br,z)=(x,y,z)$ and define $\bk=(\bka,k_3)=(\kappa_1,\kappa_2,k_3)$ where $k_3\in\lbrace \frac{2\pi}{L_3}n:n\in \mathbb{Z}\rbrace$, $\kappa_{1,2}\in\mathbb{R}$ and $\kappa^2=\kappa_1^2+\kappa_2^2$. Then the Ewald1P formula to compute the potential, denoted by $\varphi$, at a source location $\bx_{\mi}$, $\mi=1,\ldots,N$, reads
\begin{align}
\varphi(\bx_{\mi})&=\varphi^\R (\bx_{\mi})+\varphi^{\F,k_3\ne0}(\bx_{\mi})+\varphi^{\F,k_3=0}(\bx_{\mi})+\varphi^{\textrm{self}}_{\mi},
\label{eq:full_potential}
\end{align}
where
\begin{align}
\varphi^{\R}(\bx_{\mi})&=\sum_{\bp\in P_1}^{\prime}\sum_{\ni=1}^Nq_{\ni}\dfrac{\erfc(\xi|\bx_{\mi}-\bx_{\ni}+\bp|)}{|\bx_{\mi}-\bx_{\ni}+\bp|}, \label{eq:ewald1p_real}\\ 
\varphi^{\F,k_3\ne0}(\bx_{\mi})&=\dfrac{1}{\pi L_3}\sum_{k_3\neq0}\sum_{\ni=1}^N q_{\ni}e^{\ii k_3(z_{\mi}-z_{\ni})}\int_{\mathbb{R}^2} \frac{e^{-(\kappa^2+k_3^2)/4\xi^2}}{\kappa^2+k_3^2}e^{\ii\bka\cdot(\br_{\mi}-\br_{\ni})}\dif\bka \refstepcounter{equation}\subeqn \label{eq:ewald1p_fourier_integral}\\
&=\frac{1}{L_3}\sum_{k_3\neq0}\sum_{\n=1}^Nq_{\ni}e^{\ii k_3(z_{\mi}-z_{\ni})}{\bf K_0}(k_3^2/4\xi^2,\rho_{\mi\ni}^2\xi^2),\label{eq:ewald1p_fourier_complete} \subeqn\\
\varphi^{\F,k_3=0}(\bx_{\mi})&=-\dfrac{1}{L_3}\sum_\subindex{\ni=1}{\ni\neq \mi}^Nq_{\ni}\lbrace \gamma+\log(\xi^2\rho_{\mi\ni}^2)+\E(\xi^2\rho_{\mi\ni}^2)\rbrace, \label{eq:ewald1p_zero}\\
\varphi^{\textrm{self}}_{\mi}&=-\frac{2\xi}{\sqrt{\pi}}q_{\mi},
\label{eq:ewald1p_self}
\end{align}
in which $\rho_{\mi\ni}=|\br_{\mi}-\br_{\ni}|$, $\gamma=0.5772156649\ldots$ is the Euler-Mascheroni constant and $\bf K_0(\cdot,\cdot)$ is the incomplete modified Bessel function of the second kind defined as
\begin{align*}
{\bf K_0}(a,b) = \int_1^\infty\dfrac{\dif t}{t}e^{-at-b/t} = \int_0^1\dfrac{\dif t}{t}e^{-a/t-bt}.
\end{align*}
Also $\E(\cdot)$ is the exponential integral and is defined as \cite[Sec 6.3]{Press1987},
\begin{align*}
\E(x) = \int_1^{\infty}\dfrac{e^{-xt}}{t}\dif t, \quad x>0.
\end{align*}
The self interaction terms in \eqref{eq:ewald1p_self} and \eqref{eq:ewald3p} (the last terms) are identical. 
Except for the number of periodic dimensions, the real space sums in \eqref{eq:ewald1p_real} and \eqref{eq:ewald3p} (the first terms) are also the same and therefore, can be evaluated similarly by truncating the infinite sums. The $k$-space sums are different in the way that in the 1d-periodic case there is no longer a discrete summation in all three directions. In the two free directions, there is now an inverse Fourier transform instead. These integrals can be evaluated analytically, with the result in \eqref{eq:ewald1p_fourier_complete}. It is however the original form \eqref{eq:ewald1p_fourier_integral} that will be the basis for the fast method that we present. 

The third term \eqref{eq:ewald1p_zero} has no correspondence in the triply periodic case. It is the $k_3=0$ term that has been separated out from the sum over $k_3$ in \eqref{eq:ewald1p_fourier_integral}. If it were to be summed directly, it would also have an $\order{N^2}$ complexity. In the fast method, we will be able to include this term in the FFT treatment of \eqref{eq:ewald1p_fourier_integral}, using ideas from  \cite{Vico2016}, see section \ref{subsec:kspace}.

Equation \eqref{eq:ewald1p_zero} appears to have a singularity at $\rho_{\mi\ni}=0$, i.e., $\br_{\mi}=\br_{\ni}$. However one can show that $\log(x)+\E(x)+\gamma\to0$ as $x\to 0$, and therefore the $\ni=\mi$ term can simply be excluded from the sum. Moreover, due to the charge neutrality condition
\begin{align*}
\lim_{|\br|\to\infty}\sum_{\ni=1}^Nq_{\ni}\log(\xi^2\rho_{\ni}^2)=0,
\end{align*}
where $\rho_{\ni}=|\br-\br_{\ni}|$, see \cite[Appendix E]{Tornberg2014}. Using the definition of the exponential integral, we also have $\lim_{x\to\infty} \E(x)=0$. Thus $\varphi^{\F,k_3=0}$ is bounded everywhere.

The direct sums in $k$-space, \eqref{eq:ewald1p_fourier_complete} and \eqref{eq:ewald1p_zero}, will only be evaluated to provide reference solutions for the fast method. We briefly discuss the evaluation of ${\bf K_0}(\cdot,\cdot)$ and $\E(\cdot)$ in \ref{subsec:Incomp_bessel}.

%===================================================================================================================
%===================================================================================================================

%===================================================================================================================
%===================================================================================================================
\subsection{Truncation errors of the Ewald sums}
\label{sec:truncation_error}
The real \eqref{eq:ewald1p_real} and $k$-space \eqref{eq:ewald1p_fourier_integral}-\eqref{eq:ewald1p_fourier_complete} parts of the Ewald sum are infinite sums and will incur truncation errors when approximated by finite sums. Henceforth we denote by $r_\c\in\mathbb{R}^+$ a cut-off radius such that when evaluating $\varphi^\R(\bx_{\mi})$, we include only charges whose positions satisfy $|\bx_{\mi}-\bx_{\ni}+\bp|\leq r_\c$ as we sum over $\bp\in P_1$. The $k$-space sum is truncated at some maximum wave number $k_{\infty}$ such that $|\bk|\leq \frac{2\pi}{L}k_{\infty}$. Throughout this article, for simplicity we assume $L_1=L_2=L_3=L$ unless it is specified otherwise.

Similarly to what was concluded in \cite{Lindbo2012} for the doubly periodic case, in the 1d-periodic Spectral Ewald method the truncation error estimates given in \cite{Kolafa1992} for 3d-periodic problems are still valid. This is due to the fact that regardless of the periodicity, we expect the truncation errors to be the same in each direction. To measure the error, the following \textit{root mean square} error is used
\begin{align*}
\e_\textrm{rms} := \left( \dfrac{1}{N}\sum_{\ni=1}^N \left(\Delta\varphi(\bx_{\ni})\right)^2\right)^{1/2},\quad \Delta\varphi(\bx_{\ni})=(\varphi-\varphi^\ast)(\bx_{\ni}).
\end{align*}
The reference solution $\varphi^\ast$ is computed using a very well converged approximate solution or the direct sum \eqref{eq:ewald1p_fourier_complete}, with ${\bf K_0}$ computed as in \ref{subsec:Incomp_bessel} with high precision. 

The real and $k$-space truncation error estimates to compute the potential \eqref{eq:coulomb} using Ewald summation formula respectively read \cite{Kolafa1992},
\begin{align}
\e_\textrm{rms}^\R&\approx\sqrt{\dfrac{Qr_\c}{2L^3}}(\xi r_\c)^{-2}e^{-\xi^2 r_\c^2}, \label{eq:trunc_error_estimate_real} \\
\e_\textrm{rms}^\F&\approx\xi\pi^{-2}k_{\infty}^{-3/2}\sqrt{Q}e^{-(\pi k_{\infty}/\xi L)^2},
\label{eq:trunc_error_estimate_fourier}
\end{align}
where $Q:=\sum_{\ni=1}^Nq_{\ni}^2$. To assess the accuracy of the estimates above, we consider a system of $N=100$ randomly distributed particles with $Q=\order{1}$ in a cubic box of size $L=2$. In figure \ref{fig:kolafa_perram} we plot the absolute rms error in the computation of the real (left) and Fourier (right) space parts of the 1d-periodic Ewald sum for this system. Referring to these plots we find an excellent agreement between the measured errors and the error estimates. 

Assume $\xi$ to be given and that the truncation level is set to $\e=\e_\textrm{rms}^\R=\e_\textrm{rms}^\F$. By inverting the error formulas in \eqref{eq:trunc_error_estimate_real}-\eqref{eq:trunc_error_estimate_fourier}, we can compute the Ewald sum parameters, $r_\c$ and $k_{\infty}$,
\begin{align}
r_\c &\approx \dfrac{1}{2\xi}\left[ 3W\left(\dfrac{4}{3}C^{2/3}\right) \right]^{1/2}, \quad\qquad C=\dfrac{Q}{2L^3\xi\e^2}, \label{eq:rc} \\
k_{\infty} &\approx \frac{\sqrt{3}L\xi}{2\pi}\left[ W\left( \dfrac{D}{(\xi\e^2)^{2/3}} \right)\right]^{1/2}, ~~D = \dfrac{4}{3L^ 2}\left(\dfrac{Q}{\pi}\right)^{2/3}, \label{eq:kinf}
\end{align}
where we denote by $W(\cdot)$, the \textsf{LambertW} function, the inverse of $f(x)=xe^x$. 
\begin{figure}[htbp]
  \begin{subfigure}[b]{0.49\textwidth}
    \centering 
    \includegraphics[width=\textwidth,height=.8\textwidth]{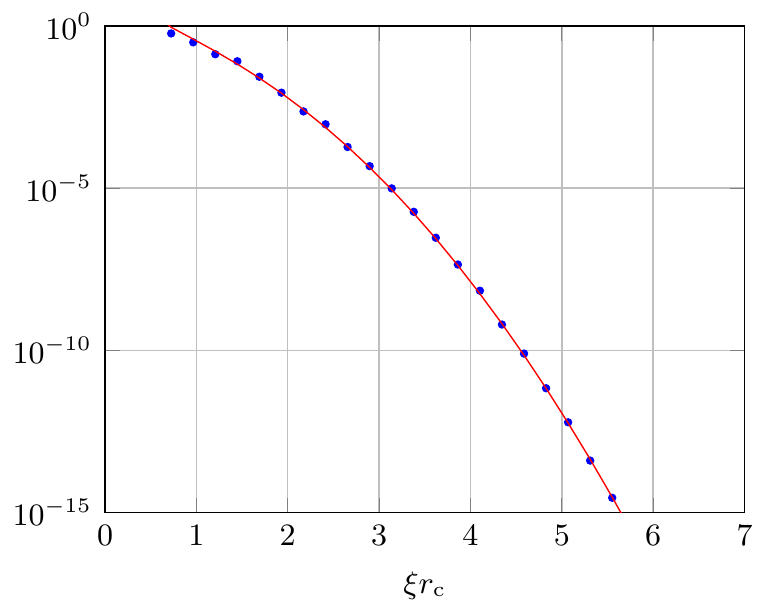}
  \end{subfigure}
  \begin{subfigure}[b]{0.49\textwidth}
    \centering 
    \includegraphics[width=\textwidth,height=.8\textwidth]{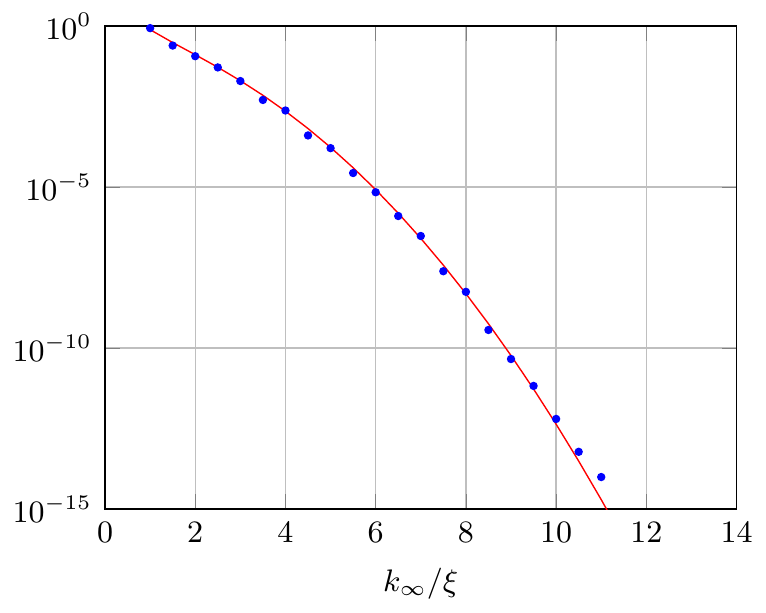}
  \end{subfigure}
  \caption{(Left) RMS of absolute truncation errors in real space with $\xi=7$ and $r_\mathrm{c}\in(0,1)$. (Right) RMS of absolute truncation errors in Fourier space with $\xi=3.14$ and $k_{\infty}\in[2,30]$. A system of $N=100$ randomly distributed particles in a cubic box of size $L=2$ is used. Dots are measured values and solid lines are computed using the estimates \eqref{eq:trunc_error_estimate_real} and \eqref{eq:trunc_error_estimate_fourier}.}
  \label{fig:kolafa_perram}
\end{figure}

%===================================================================================================================
%===================================================================================================================

%===================================================================================================================
%===================================================================================================================
\section{Introducing a fast method}
\label{sec:se1p}
According to the discussion in the previous section, contributions to the real space sum will be ignored if the distance between the source location and the target evaluation point is larger than a cut-off radius $r_\c$. Typically, a linked cell list \cite{Allen1989, Hockney2010} or a Verlet list algorithm \cite{Allen1989,Verlet1967} can be used to efficiently obtain a list of nearest neighbors. The real space sum in \eqref{eq:ewald1p_real} includes a summation over the one periodic dimension, which means that contribution from periodic images of the sources are also included if they are within this distance. This is similar to the triply periodic case, where periodic distances must be checked in all three directions. 
Hence, any efficient implementation for the triply periodic real space sum can simply be modified to instead compute the singly periodic real space sum. We will therefore not discuss the evaluation of the real space sum further in this paper. 

In the reminder of this section, we will derive the formulas and introduce the approximations needed to build a method for the rapid and accurate evaluation of the two terms \eqref{eq:ewald1p_fourier_integral} and \eqref{eq:ewald1p_zero} by FFTs. We will start by schematically describing the algorithm, later filling in the details.

%===================================================================================================================
%===================================================================================================================

%===================================================================================================================
%===================================================================================================================
\subsection{The k-space algorithm}
\label{subsec:kspacealg}

In this section, we will introduce the main steps of the $k$-space algorithm to convey its structure.
%such that it is possible to relate is structure to methods in the PME family for the triply periodic problem. The section %can however be read also without a familiarity of such methods. 
Derivations of formulas, discussion of parameter choices and associated errors will follow. 

Let us denote 
\begin{align}
\varphi^\F(\bx)=\varphi^{\F,k_3\ne0}(\bx)+\varphi^{\F,k_3=0}(\bx),
\label{eq:sum_of_phis}
\end{align}
where $\varphi^{\F,k_3\ne0}$ and $\varphi^{\F,k_3=0}$ are given in \eqref{eq:ewald1p_fourier_integral} and \eqref{eq:ewald1p_zero}, respectively. The objective is to evaluate $\varphi^\F(\bx)$ 
at given locations (target points) accurately and efficiently in such a way that the computation of the $\varphi^{\F,k_3\ne0}(\bx)$ and $\varphi^{\F,k_3=0}(\bx)$ terms are unified. 

For simplicity of description, assume that $L_1=L_2=L_3=L$. 
We first introduce a uniform grid of size $M^3$ on $[0,L)^3$ with grid size $h=L/M$. In any method of Particle Mesh Ewald (PME) type, one would start by spreading point charges to the grid by interpolation. 
As will be discussed later, in the SE method we use a suitably scaled and truncated Gaussian as this ``interpolation'', ``spreading'' or ``window'' function. 
The charge locations are all in the domain $[0,L)^3$. The domain length in the free directions must be extended to also contain the support of a truncated Gaussian centered around any possible charge location. In the periodic direction, they will instead be wrapped around periodically. With a support of the truncated Gaussian of $P^3$ points, 
denote the extended domain length by $\tilde{L}=L+Ph$ and the number of grid points in the free directions by $\tilde{M}=M+P$ s.t. $h=L/M=\tilde{L}/\tilde{M}$. 

Applying the FFT to any function defined on this grid would yield the Fourier coefficients for the $k$-space vectors $(2\pi n_1/\tilde{L}, 2\pi n_2/\tilde{L}, 2\pi n_3/L)$ with $n_i=-\tilde{M}/2,\ldots,\tilde{M}/2-1$, $i=1,2$ and $n_3=-M/2,\ldots,M/2-1$.  Hence, the maximum absolute value of each component will be the same. The problem is however only periodic in the $z$ coordinate, and \eqref{eq:ewald1p_fourier_integral} contains a discrete sum over $k_3$ and an integral over $\kappa_1$ and $\kappa_2$. When we discretize the integral, we will obtain discrete sums also over $\kappa_1$ and $\kappa_2$. As will be discussed later, we will however need a finer resolution in the $\kappa_1$ and $\kappa_2$ directions for some $k_3$ modes to obtain an accurate approximation of the integral.  A finer resolution in $k$-space is achieved by zero-padding in real space. That is, if we extend the grid from $\tilde{M}$ to $s\tilde{M}$ points in one direction, the spacing in $k$-space decreases by a factor $s$ to $2\pi/(s\tilde{L})$ while the maximum magnitude of the $k$-modes stays the same. The parameter $s$ is called the \textit{upsampling factor}.

Given an integer $\nl$, we can define the following \textit{local pad} set 
\begin{align}
\mathbb{I}:=\lbrace k_3:  0 < |k_3|\leq  \frac{2\pi}{L} \nl \rbrace,
%%\ll  \frac{2\pi}{L} k_{\infty}\rbrace.
\label{eq:I}
\end{align}
where $\nl \ll k_{\infty}$. 
Note that $|\mathbb{I}|=2\nl$.
For the rest of the  non zero $k_3$ modes, we define
\begin{equation}
\mathbb{J}=\lbrace k_3:  \frac{2\pi}{L} \nl  <|k_3| \leq  \frac{2\pi}{L} k_{\infty}\rbrace.
\label{eq:J}
\end{equation}
%%Given $n_\ell$ needed to define the sets above, as well as 
Given upsampling rates $s_0$ and $\sl $,
we indicate by AFT the \textit{adaptive Fourier transform} that computes the Fourier transform $(x,y) \rightarrow (\kappa_1,\kappa_2)$ for $k_3=0$ with an oversampling factor $s_0$ and a resulting $\kappa$-spacing of $2\pi/(s_0 \tilde{L})$, for all $k_3 \in \mathbb{I}$ with an oversampling factor $\sl $ and a resulting $\kappa$-spacing of $2\pi/(\sl  \tilde{L})$, and for the remaining $k_3$ modes ($k_3 \in \mathbb{J}$) without oversampling which yields the basic $\kappa$-spacing $2\pi/\tilde{L}$, i.e.\ we define the adaptive upsampling factor 
\begin{align}
s=\left\lbrace
\begin{array}{cl}
s_0, & k_3=0, \\
\vspace{-.3cm}\\
\sl , & |k_3|\leq\frac{2\pi}{L}\nl, \\
\vspace{-.3cm}\\
1, & |k_3|>\frac{2\pi}{L}\nl.
\end{array}
\right.
\label{eq:upsamplings}
\end{align}
The implementation of the AFT will be discussed in section \ref{sec:implementation}. Note that if $s_0=\sl $ and $\nl=k_\infty$, then all modes are oversampled with a rate $s=s_0=\sl $ and we are back at the plain upsampled FFT.

We have yet to discuss the different parameters. 
The choice of $M$ is related to the truncation error estimate for the Ewald $k$-space sum
\eqref{eq:trunc_error_estimate_fourier}, with $M/2=k_{\infty}$, and will depend on the decomposition parameter $\xi$ and the error tolerance. The approximation errors that arise from the introduction of the fast method and the selection of remaining parameters will be discussed in section \ref{sect:approxerr_parsel}. 
The truncated Gaussians will be scaled to minimize the approximation error given the number of grid points $P$ in the support across the Gaussian, such that $P$ is the only parameter to select. 

We state the full algorithm below in Algorithm \ref{alg:se1p}. 
{
\small
\begin{breakablealgorithm}
\caption{1d-periodic spectral Ewald (SE1P) method - Fourier space part}
\small
\label{alg:se1p}
\begin{algorithmic}[1]
\INPUT Charge locations $\bx_{\ni}\in[0,L)^3$ and charges $q_{\ni}$, $\ni=1,\ldots,N$, splitting parameter $\xi$, grid size $M$, oversampling factors $s_0, \sl $, maximum oversampled wavenumber $\nl$, number of points $P$ in support of truncated Gaussians.
\State Set $h=L/M$, $\tilde{L}=L+Ph$, $\tilde{M}=M+P$. Compute
\begin{align}
\eta=\frac{P\xi^2 h^2}{c^2 \pi},
\label{eq:comp_eta}
\end{align}
where $c=0.95$.
\State 
Introduce a uniform grid on $[0,L)\times [0,\tilde{L}) \times [0,\tilde{L})$ with $M \times \tilde{M} \times \tilde{M}$ points. Evaluate $H$ on this grid according to
\begin{align}
H(\br,z) = \left(\dfrac{2\xi^2}{\pi\eta}\right)^{3/2}\sum_{\ni=1}^Nq_{\ni}e^{-2\xi^2|\br-\br_{\ni}|^2/\eta}e^{-2\xi^2(z-z_{\ni})_{\ast}^2/\eta},
\label{eq:spread_1p}
\end{align}
where $(\cdot)_{\ast}$ denotes the closest periodic distance in the $z$ direction and where the Gaussian is truncated outside of a support of $P^3$ points centered around $(\br_{\ni}, z_{\ni})$.
\State Apply an AFT with $s_0$, $\sl $ and $\nl$ to compute $\wh{H}(\kappa_1,\kappa_2,k_3)$. 
\State Scale $\wh{H}$ to obtain $\wh{\wt{H}}$ 
%%as in \eqref{eq:scale_1p}.\;
\begin{align}
\wh{\wt{H}}(\kappa_1,\kappa_2,k_3):=
%%\dfrac{e^{-(1-\eta)(\kappa^2+k_3^2)/4\xi^2}}{\kappa^2+k_3^2}\wh{H}(\bka,k_3).
e^{-(1-\eta)(\kappa^2+k_3^2)/4\xi^2} \wh{G}(\kappa,k_3) \wh{H}(\kappa_1,\kappa_2,k_3),
\label{eq:scale_1p}
\end{align}
where $\kappa^2=\kappa_1^2+\kappa_2^2$ and  
\begin{align}
\wh{G}(\kappa,k_3) = \left\lbrace
\begin{array}{ll}
\dfrac{1}{\kappa^2+k_3^2},& k_3\neq0,\\
\vspace{-.25cm}\\
\dfrac{1-J_0(R\kappa)}{\kappa^2}-\dfrac{R\log(R)J_1(R\kappa)}{\kappa},& k_3=0, \kappa\neq0,\\
\vspace{-.25cm}\\
\dfrac{R^2}{4}(1-2\log(R)),& k_3=0, \kappa=0,\\
\end{array}
\right. 
\label{eq:G_kappa_k}
\end{align}
and $R=\sqrt{2} \tilde{L}$.
\State Apply an inverse AFT to $\wh{\wt{H}}$ to obtain $\wt{H}$ on the grid of size 
$M \times \tilde{M} \times \tilde{M}$.
%%as in \eqref{eq:imft_trap}\;
\State Compute the potential at target points (same as charge locations)
\begin{align}
\varphi^\F(\br_\ni,z_\ni) &\approx \dfrac{2}{L_3}\left(\frac{2\xi^2}{\pi\eta}\right)^{\frac{3}{2}}h^3\sum_{n,m,l}\wt{H}(nh,mh,lh)e^{-2\xi^2|\br_\ni-(nh,mh)|^2/\eta}e^{-2\xi^2(z_\ni-lh)_{\ast}^2/\eta}.
\label{eq:se1p_complete_discretized}
\end{align}
As in step 2, the Gaussians are truncated outside of a support of $P^3$ points.
\OUTPUT Potentials $\varphi^\F(\bx_{\ni})=\varphi^{\F,k_3\ne0}(\bx_{\ni})+\varphi^{\F,k_3=0}(\bx_{\ni})$, $\ni=1,\ldots,N$.
\end{algorithmic}
\end{breakablealgorithm}
}
Step 2 in Algorithm \ref{alg:se1p} is referred to as the \textit{gridding} step, when truncated Gaussians centered at the charge locations are evaluated on a uniform grid. Step 5, when grid values are known and the same truncated Gaussians are used to interpolate to the target locations, is referred to as the \textit{gathering} step. For both these steps, we use Fast Gaussian Gridding (FGG) to obtain an efficient implementation, as will be discussed in section \ref{sec:implementation}.

The factor $e^{-(1-\eta)(\kappa^2+k_3^2)/4\xi^2} $ in \eqref{eq:scale_1p} (the \textit{scaling} step) depends on the specific choice of the Gaussian as the window function. 
The definition of $\wh{G}(\kappa,k_3)$ in \eqref{eq:G_kappa_k} is a key new component of this work and is based on a very recent method for how to solve free space problems with FFTs \cite{Vico2016}. This term will be derived in section \ref{subsec:kspace}. As we can clearly see at this point, this approach really unifies the treatment of the $k_3=0$ mode with the modes for $k_3 \ne 0$. All that is needed is a different scaling for this mode, together with a specific choice of oversampling factor $s_0$. 

The insight that only some $k_3$ modes need an upsampled 2D FFT (as will be discussed in section \ref{sec:spectral_accuracy}), and implementation of the adaptive FFT (AFT) to utilize this fact
(section \ref{sec:implementation}), further really enhances the efficiency of the method.

The recent FFT-based method by Nestler \etal \cite{Nestler2015} uses a different approach. The main idea is to start directly from  \eqref{eq:ewald1p_fourier_complete} and approximate 
\begin{align*} 
B(k_3,\br_{\mi},\br_{\ni}) = \bessel(k_3^2/4\xi^2,|\br_{\mi}-\br_{\ni}|^2\xi^2),
\end{align*}
using Fourier series. The functions $B(k_3,\cdot,\cdot)$ are smooth but not periodic in the $x$ and $y$ directions. The simulation box is doubled in both directions and $B(k_3,\cdot,\cdot)$ is truncated on the extended interval. However, this extension is not adequate for the Fourier approximation to be done since the Fourier coefficients do not decay sufficiently fast. To enforce periodicity and some degree of smoothness, the simulation box is extended more and a smooth transition is constructed on the resulting gap using polynomials of degree $2p-2$, where $p$ is the degree of smoothness. Applying a similar strategy on the zero mode \eqref{eq:ewald1p_zero}, a unified regularization is obtained. Moreover, by construction, the regularization functions are periodic and $C^{p-1}$ in both $x$ and $y$ directions. Therefore, the functions $B(k_3,\cdot,\cdot)$ for all $k_3$ and the zero mode (corresponding to $k_3=0$) can be approximated using Fourier series in which the coefficients are computed using regularization functions. These Fourier coefficients are computed on a uniform grid. Therefore, a non uniform FFT (NFFT) is required to spread arbitrarily distributed point particles onto a uniform grid and  compute Fourier coefficients. Another NFFT is also required to take the scaled Fourier transformed point particles back to the real space. To compute the force, the differentiation operator is applied to the window function in the NFFT. We remark that the window function in the NFFT algorithm and the Gaussians that we use in the gridding and gathering steps have a similar functionality. %But the term ``upsampling factor'' defined in our algorithm differs fundamentally with the similar factor defined in the context of non uniform FFT. 

In section \ref{sec:results} (example \ref{ex:wall}) we provide a numerical result that compares efficiency of the SE1P and singly periodic NFFT-based methods with their triply periodic counterparts.

%===================================================================================================================
%===================================================================================================================

%===================================================================================================================
%===================================================================================================================
\subsection{The k-space formulas}
\label{subsec:kspace}
In this section we present derivations that yield the definition of the modified Green's function in \eqref{eq:G_kappa_k}. The function $\varphi^\F(\bx)$ in \eqref{eq:sum_of_phis} is the solution to the 3D Poisson equation \cite{Tornberg2014},
\begin{align}
-\Delta\varphi^\F(\bx) = f(\bx)=4\pi(f^{1\P}*\tau)(\bx), \quad f^{1\P}(\bx) = \sum_{\bp\in P_1}\sum_{\ni=1}^Nq_{\ni}\delta(\bx-\bx_{\ni}+\bp),
\label{eq:Poisson3D}
\end{align}
in which $\tau(\xi,\bx)$ is a screening function used to obtain the classical Ewald sum decomposition,
\begin{align*}
\tau(\bx) = \pi^{-3/2}\xi^3e^{-\xi^2|\bx|^2} \mbox{ with } \hat{\tau}(\kappa_1,\kappa_2,k_3) = e^{-(\kappa_1^2+\kappa_2^2)/4\xi^2}e^{-k_3^2 /4\xi^2}.
\end{align*}
Here we remind about the notation $\bx=(\br,z)=(x,y,z)$,  $\bk=(\bka,k_3)=(\kappa_1,\kappa_2,k_3)$ and  $\kappa^2=\kappa_1^2+\kappa_2^2$. 

Expanding $\varphi^\F(\bx)$ as a Fourier series in $z$ where the Fourier modes $k_3$ form the discrete set 
$\lbrace\frac{2\pi}{L_3} n:  n \in \mathbb{Z} \rbrace$ we get  
\begin{equation}
\varphi^\F(\bx)=\sum_{k_3} \phikthree(\br) e^{\ii k_3 z}.
\label{eq:phi_fourier_z}
\end{equation}
Expanding also $f(\bx)$ similarly, inserting into \eqref{eq:Poisson3D} and using orthogonality yields
\begin{align}
(-\Delta_{2D}+k_3^2)\phikthree(\br)=4\pi \fkthree(\br), \quad k_3\in\mathbb{Z},
\label{eq:poisson_zeromod0}
\end{align}
in which $\Delta_{2D}$ denotes the two dimensional Laplacian operator and 
\begin{equation}
\fkthree(\br)=\frac{1}{L_3} \sum_{\ni=1}^N q_{\ni} \left(  \pi^{-1}\xi^2e^{-\xi^2|\br-\br_{\ni}|^2} \right) e^{-k_3^2/4\xi^2}e^{-\ii k_3z_{\ni}}.
\label{eq:fkthree_of_r}
\end{equation}
The Fourier coefficients $\phikthree(\br)$ can be represented in terms of a Fourier transform in the non-periodic 
directions $x$ and $y$, 
\begin{align*}
\phikthree(\br)=\frac{1}{(2\pi)^2} \int_{\reals^2} \wh{\varphi}_{k_3}(\bka) e^{\ii \bka\cdot \br}  \dif\bka,
\end{align*}
and similarly for  $\fkthree(\br)$. For $\fkthree(\br)$ we write the inverse relation as
\begin{align*}
\hat{f}_{k_3}(\bka)=\int_{\reals^2} \fkthree(\br) e^{-\ii \bka\cdot \br}  \dif\br=
\frac{1}{L_3} \sum_{\ni=1}^N q_{\ni} e^{-\kappa^2/4\xi^2} e^{-k_3^2/4\xi^2} e^{-\ii \bka \cdot \br_{\ni}}  e^{-\ii k_3z_{\ni}}.
\end{align*}
Considering \eqref{eq:poisson_zeromod0} for $k_3 \ne 0$, the relation is
\begin{align*}
 \wh{\varphi}_{k_3}(\bka) =4\pi \frac{1}{\kappa^2+k_3^2}  \hat{f}_{k_3}(\bka), 
\end{align*}
and with this
\begin{align*}
\phikthree(\br)=\frac{1}{\pi L_3} \sum_{\ni=1}^N q_{\ni}  e^{-\ii k_3z_{\ni}} e^{-k_3^2/4\xi^2} 
\int_{\reals^2}  \frac{e^{-\kappa^2/4\xi^2}}{\kappa^2+k_3^2} e^{\ii \bka\cdot (\br-\br_{\ni})}  \dif\bka. 
\end{align*}
Inserting into \eqref{eq:phi_fourier_z} and summing up excluding $k_3=0$, we get $\varphi^{\F,k_3\ne0}$ as in \eqref{eq:ewald1p_fourier_integral}. 
The term $\varphi^{\F,k_3=0}$ is given by the remaining coefficient,
$\varphi^{\F,k_3=0}(\bx)=\phinull(\br)$, which is the solution to \eqref{eq:poisson_zeromod0} for $k_3=0$. 
We can however not proceed in the same way as for $k_3 \ne 0$ as 
it would render a singular integral. Directly solving  \eqref{eq:poisson_zeromod0} for $k_3=0$ one obtains $\varphi^{\F,k_3=0}(\bx)$ on the form 
\begin{align}
\varphi^{\F,k_3=0}(\bx) &= 
-\dfrac{1}{L_3}\sum_{\ni=1}^N q_{\ni}\left\lbrace  \log(|\br-\br_{\ni}|^2)+\E(\xi^2  |\br-\br_{\ni}|^2 )\right\rbrace,
\label{eq:k0_term}
\end{align}
which has the finite limit
\begin{align*}
\lim_{r \rightarrow 0} \left\{ \log(r^2)+\E(\xi^2 r^2) \right\} = \gamma+\log(\xi^2),
\end{align*}
and is equivalent to the zero mode term in the Ewald sum \eqref{eq:ewald1p_zero} by the charge neutrality of the system. 

We can also compute $\varphi^{\F,k_3=0}(\bx)=\phinull(\br)$ by convolving the right hand side with the Green's function or fundamental solution of the problem,
which again can be restated in the Fourier domain, 
\begin{align}
\phinull(\br)= \int_{{\mathbb{R}}^2} G(\br-\bs) \fknull(\bs)\dif\bs =
\frac{1}{(2\pi)^2} \int_{\reals^2}  \wh{G}(\bka) \hat{f}_0(\bka) e^{\ii \bka\cdot \br}  \dif\bka,
\label{eq:phinull}
\end{align}
where for the 2D Poisson equation we have
\begin{align*}
G(\br)=-\frac{1}{2\pi}\log(|\br|). 
\end{align*}
Here, again we have $ \wh{G}(\bka)=\frac{1}{\kappa^2}$, and we get exactly the same as we would get above if we proceeded for $k_3=0$ in the same manner as for $k_3 \ne 0$. 

Let us now assume that $f_0(\br)$ \eqref{eq:fkthree_of_r} is compactly supported in $\tilde{\Omega}=\{ (x,y): 0 \le x,y \le \tilde{L} \}$ and that we seek the solution $\phinull(\br)$ inside this domain. Let $R=\sqrt{2}\tilde{L}$ be the maximum distance between two points in the domain. The Gaussians $e^{-\xi^2|\br-\br_{\ni}|^2}$ in \eqref{eq:fkthree_of_r} do not formally have compact support, but can in practice be truncated once sufficiently decayed.  
Assuming that the Gaussians are truncated at a radius $w$, then $\tilde{L}=L+2w$.  

Notice that $G(\br)=G(r)=-\frac{1}{2\pi}\log(r)$, where $r=|\br|$. Following \cite{Vico2016},  we define a truncated version of $G$ as
\begin{align}
G_\subR(r) := G(r)\cdot \text{rect}\left(\dfrac{r}{2R}\right),
\label{eq:truncated_g}
\end{align}
where
\begin{align*}
\text{rect}(x) = \left\lbrace
\begin{array}{cc}
1,& |x|\le 1/2,\\
0,& |x|>1/2,\\
\end{array}
\right. 
\end{align*}
and we have
\begin{align*}
\int_{{\mathbb{R}}^2} G(\br-\bs) \fknull(\bs)\dif\bs =\int_{{\mathbb{R}}^2} G_R(\br-\bs) \fknull(\bs)\dif\bs, \quad  \br \in \tilde{\Omega}.
\end{align*}
The Fourier transform of $G_\subR$  is radially symmetric since $G_\subR$ is so, and can be computed as \cite{Vico2016},
\begin{align}
\wh{G}_\subR(\kappa) = 2\pi\int_0^{\infty} J_0(\kappa r) G_\subR(r)r\dif r &= 2\pi\int_0^{R} J_0(\kappa r) G(r)r\dif r \nonumber \\
&= \dfrac{1-J_0(R\kappa)}{\kappa^2}-\dfrac{R\log(R)J_1(R\kappa)}{\kappa},
\label{eq:g_zeromod}
\end{align}
where $J_n(\cdot)$ is the $n$th order Bessel function of the first kind. 
Moreover, $\wh{G}_\subR(\kappa)$ has a finite limit at $\kappa=0$,

\begin{align}
\lim_{\kappa \rightarrow 0} \wh{G}_\subR(\kappa) = \dfrac{R^2}{4}(1-2\log(R)).
\label{eq:g_zeromod0}
\end{align}
We can now replace $\wh{G}(\bka)$ by $\wh{G}_\subR(\kappa)$ in the definition of 
$\phinull$ in \eqref{eq:phinull}. We now have a definition of $\phikthree$ for all $k_3$. 
Inserting this into the Fourier sum in \eqref{eq:phi_fourier_z}, we finally arrive at
\begin{align}
\varphi^\F(\bx) & =\varphi^{\F,k_3\ne0}(\bx)+\varphi^{\F,k_3=0}(\bx) \notag \\
& = \frac{1}{\pi L_3} \sum_{k_3} \sum_{\ni=1}^N q_{\ni} e^{\ii k_3 (z-z_{\ni})} e^{-k_3^2/4\xi^2} 
\int_{\reals^2}  \wh{G}(\kappa,k_3) e^{-\kappa^2/4\xi^2}  e^{\ii \bka\cdot (\br-\br_{\ni})}\dif\bka,
\label{eq:ewald1p_allk}
\end{align}
for $\bx$ in the simulation box and with $\wh{G}$ as defined in \eqref{eq:G_kappa_k}.

%===================================================================================================================
%===================================================================================================================

%===================================================================================================================
%===================================================================================================================
\subsection{Spreading, scaling and gathering}
\label{subsec:foundation}
In this section, we introduce a split of the Gaussian factors in \eqref{eq:ewald1p_allk} to derive the formulas for the spreading, scaling and gathering steps in Algorithm \ref{alg:se1p}. Consider again the definition of $\varphi^\F(\bx)$ in \eqref{eq:sum_of_phis}. Introducing $\eta>0$, we split the Gaussian term in three parts, 
\begin{align}
e^{-(\kappa^2+k_3^2)/4\xi^2}=e^{-(1-\eta)(\kappa^2+k_3^2)/4\xi^2}e^{-\eta(\kappa^2+k_3^2)/8\xi^2}e^{-\eta(\kappa^2+k_3^2)/8\xi^2}.
\label{eq:splitting_1p}
\end{align}
With this we can write 
\begin{align}
\varphi^\F(\bx) &= \dfrac{1}{\pi L_3}\sum_{k_3}\int_{\rtwo}
e^{-(1-\eta)(\kappa^2+k_3^2)/4\xi^2}\wh{G}(\kappa,k_3)
e^{\ii\bka\cdot\br}e^{\ii k_3z}e^{-\eta(\kappa^2+k_3^2)/8\xi^2}\overline{\wh{H}(\bka,k_3)}\dif\bka,
\label{eq:split_ewald_1p}
\end{align}
where we have defined
\begin{align}
\wh{H}(\bka,k_3):=\sum_{\ni=1}^N q_{\ni}e^{-\eta(\kappa^2+k_3^2)/8\xi^2}e^{\ii \bka\cdot\br_\ni}e^{\ii k_3z_{\ni}}.
\label{eq:Hhat_ewald_1p}
\end{align}
Continuing from \eqref{eq:Hhat_ewald_1p}, applying the convolution and Parseval theorems, the inverse Fourier transform of \eqref{eq:Hhat_ewald_1p} can be written as in \eqref{eq:spread_1p}, step 2 in Algorithm \ref{alg:se1p} (although without truncation of Gaussians).
For this derivation, we have used the fact that $e^{-\eta k^2/8\xi^2} $ is the Fourier transform of $\left(\frac{2\xi^2}{\pi\eta}\right)^{3/2}e^{-2\xi^2|\bx|^2/\eta}$ and $e^{\ii \bk\cdot \bx_{\ni}}$ is the Fourier transform of $\delta(\bx-\bx_{\ni})$ where $\delta(\cdot)$ is the Dirac delta function. 
The function $e^{-2\xi^2|\bx|^2/\eta}$ in \eqref{eq:spread_1p} serves as our window function. The Gaussians  are smooth and minimize aliasing errors, however, they do not have compact support and therefore, have to be truncated in practice. Utilizing a Gaussian as the window function is a key point in our algorithm. In fact, the accuracy of the FFT-based quadrature method, which we present here, relies heavily on the regularity of the window function and the accuracy drops considerably whenever a less regular window function, e.g., cardinal B-splines in the SPME method, is used, {\cite{Lindbo2012}}. We shall return to this discussion in section \ref{subsec:approximation_error}.

Considering equation \eqref{eq:split_ewald_1p} the Fourier coefficients $\wh{H}$ are scaled as in \eqref{eq:scale_1p} and therefore \eqref{eq:split_ewald_1p} is rewritten as
\begin{align}
\varphi^\F(\bx,\xi) = \varphi^\F(\br,z,\xi) &= \dfrac{1}{\pi L_3}\sum_{k_3}\int_{\rtwo}e^{\ii\bka\cdot\br}e^{\ii k_3z}e^{-\eta(\kappa^2+k_3^2)/8\xi^2}\overline{\wh{\wt{H}}(\bka,k_3)}\dif\kappa_1\dif\kappa_2.
\label{eq:int_1p}
\end{align}
Given $\wh{\wt{H}}$, $\wt{H}$  is defined through an inverse mixed Fourier transform. Applying the Plancherel and convolution theorems, equation \eqref{eq:int_1p} evaluated at a target point $\bx_{\mi}=(\br_{\mi},z_{\mi})$ can be written as
\begin{align}
\varphi^\F(\bx_{\mi}) &= \dfrac{2}{L_3}\left(\dfrac{2\xi^2}{\pi\eta}\right)^{3/2}\int_{\rtwo}\int_{\Omega}\wt{H}(\br,z)e^{-2\xi^2|\br_{\mi}-\br|^2/\eta}e^{-2\xi^2(z_{\mi}-z)_{\ast}^2/\eta}\dif z\dif \br.
\label{eq:se1p_complete}
\end{align}
Equation \eqref{eq:se1p_complete} can then be discretized using the trapezoidal rule to obtain \eqref{eq:se1p_complete_discretized} in step 6 of Algorithm \ref{alg:se1p}.

%===================================================================================================================
%===================================================================================================================

%===================================================================================================================
%===================================================================================================================
\subsection{Discretizations}
\label{sec:discretization}
Given the input parameters for Algorithm \ref{alg:se1p}, $\eta$ is set according to \eqref{eq:comp_eta} in step 1 of the algorithm. In step 2, we evaluate $H$ as defined in \eqref{eq:spread_1p} on the uniform grid using truncated Gaussians. In step 3, we then need to compute the mixed Fourier transform
\begin{align}
\wh{H}(\kappa_1,\kappa_2,k_3)=\sum_{l=0}^{M-1}e^{-\ii k_3 l h}\int_{\mathbb{R}^2}H(x,y,lh)e^{-\ii \bka\cdot\br} \dif x\dif y. 
\label{eq:mft}
\end{align} 
This integral can be approximated by 
\begin{align}
\wh{H}(\kappa_1,\kappa_2,k_3) \approx h^2\sum_{n,m,l}H(nh,mh,lh)e^{-\ii (\kappa_1 nh+\kappa_2 mh+k_3lh)},
\label{eq:mft_trap} 
\end{align}
in which $l\in\lbrace0,1,\ldots,M-1\rbrace$ and $n,m\in\lbrace0,1,\ldots,\tilde{M}-1\rbrace$.
Approximation errors are introduced both due to the truncation of the Gaussian and the integration by the trapezoidal rule. In section \ref{subsec:approximation_error} we show that the errors decay spectrally in the number of points in the support of each Gaussian as $\eta$ is chosen to balance the two errors.

The scaling in \eqref{eq:scale_1p}, step 4 of Algorithm \ref{alg:se1p}, is straight forward. Then, in step 5, we need to compute the inverse mixed Fourier transform 
\begin{align}
\wt{H}(x,y,z)=\dfrac{1}{(2\pi)^2}\sum_{k_3}e^{\ii k_3z}\int_{\mathbb{R}^2}\wh{\wt{H}}(\kappa_1,\kappa_2,k_3)e^{\ii \bka\cdot\br} \dif\kappa_1\dif\kappa_2, 
\label{eq:imft}
\end{align} 
for $x,y,z$ values on the uniform grid, i.e., $(x,y,z)=(nh,mh,lh)$. We again approximate this integral with the trapezoidal rule
\begin{align}
\wt{H}(x,y,z) \approx 
\dfrac{1}{(2\pi)^2}
\sum_{k_3}(\Delta\kappa)^2e^{\ii k_3z}\sum_{\kappa_1,\kappa_2}\wh{\wt{H}}(\kappa_1,\kappa_2,k_3)e^{\ii (\kappa_1 x+\kappa_2 y)},
\label{eq:imft_trap} 
\end{align}
where $\Delta\kappa=\frac{2\pi}{s\tilde{L}}$ with $s\geq 1$ and $\kappa_{1,2}\in\Delta \kappa\lbrace-\frac{s\tilde{M}}{2},\ldots,\frac{s\tilde{M}}{2}-1\rbrace$ and $k_3\in\frac{2\pi}{L_3}\lbrace-\frac{M}{2},\ldots,\frac{M}{2}-1\rbrace$. We remind that the oversampling factor $s$ depends on $k_3$, cf.\ \eqref{eq:upsamplings}.

The maximum value of $\kappa_{1,2}$ in \eqref{eq:imft_trap} is $\frac{2\pi}{\tilde{L}} \frac{\tilde{M}}{2}$. 
The choice of $\tilde{M}$, or rather $M$, is hence related to the truncation error estimate for the Ewald $k$-space sum, with $\frac{M}{2}=k_{\infty}$. 
Approximating the integral in \eqref{eq:mft} by the trapezoidal rule to get \eqref{eq:mft_trap} yields similar requirements on the 
resolution as the approximation introduced in  \eqref{eq:se1p_complete_discretized}, and the error will depend on the truncation level of the Gaussians and how well they are resolved on the grid. 

Consider now approximating the integral in \eqref{eq:imft} by the trapezoidal rule to obtain \eqref{eq:imft_trap}. 
For the case where $k_3\ne0$, the scaling factor $\wh{G}$ in $\wh{\wt{H}}$ contains the factor $(\kappa^2+k_3^2)^{-1}$, cf.\ \eqref{eq:G_kappa_k}. The $k_3$ values form a discrete set, and for smaller values of $|k_3|$, this factor will introduce a rapid variation around $\kappa=0$, making the integral more difficult to resolve. 
It is for this reason that we will need to increase the resolution of the discretization by a factor of $s$, i.e., oversample by a factor of $s$ in $k$-space. The integral is moderately oscillatory and is damped by the rapid decay of $\wh{\wt{H}}(\kappa_1,\kappa_2,k_3)$, so for larger values of $|k_3|$ there is really no need for oversampling. In section \ref{sec:spectral_accuracy} (Theorem \ref{theo:spectral_accuracy}) we will address the necessity of oversampling for small values of $|k_3|$ and discuss the relation of oversampling factor and approximation error. 

For $k_3=0$ there is again a need of oversampling to resolve $\wh{G}(\kappa,0)$. 
It is hence unnecessary to oversample for all $k_3$ modes, as it would only introduce an extra computational cost. We have therefore introduced different oversampling rates for different values of $k_3$, cf.\ \eqref{eq:upsamplings}. The choices of which modes to oversample and by which factor depends on the accuracy requirement and will be further discussed in section \ref{subsec:parameter_selection}.

The oversampling rate that is required for a certain $k_3$ mode dictates the amount of zero padding used in the  2D FFT for that particular mode when evaluating \eqref{eq:mft_trap}. Note also that the multiplication of the factor $h^2$ in \eqref{eq:mft_trap} and $\frac{(\Delta \kappa)^2}{(2\pi)^2}$ in \eqref{eq:imft_trap} cancels, considering a built-in factor of $\frac{1}{(s\tilde{M})^2}$ for the 2D IFFT. Therefore, no scaling is made to compute the $k_3=0$ mode. A similar argument is valid for $k_3\neq0$.

In addition to the parameters defined in the 3d-periodic spectral Ewald \cite{Lindbo2011}, we have introduced extra parameters here ($s_0$, $\sl $ and $\nl $) that have to be chosen carefully to attain a given error tolerance and a reasonable speedup. In \cite{Vico2016} it is suggested that $s_0=4$ is required to resolve the Fourier transform of the truncated Green's function \eqref{eq:truncated_g} and to be able to compute aperiodic convolutions by FFT. In \cite{Klinteberg2016} we show that the minimal oversampling factor is smaller than this, $s_0=1+\sqrt{2}$. We shall also see that since $\nl $ is small compared to the grid size, the cost of applying an oversampling factor $\sl >1$ is relatively cheap, see figure \ref{fig:compare_1p3p_nestler} (left). In section \ref{subsec:parameter_selection} we will discuss how to select these parameters.

%===================================================================================================================
%===================================================================================================================

%===================================================================================================================
%===================================================================================================================
\subsection{Evaluation of the energy and force}
Besides the potential, calculation of other relevant quantities such as energy and force is of great interest in MD simulations. Since
\begin{align*}
E = \sum_{\mi=1}^Nq_{\mi}\varphi(\bx_{\mi}),
\end{align*}
the corresponding Ewald1P formula to compute energy can be obtained by multiplication of the electrostatic potential $\varphi(\bx_{\mi})$ with $q_{\mi}$ and a summation over $\mi$.

The electrostatic force exerted on each particle by other particles is given by
\begin{align}
{\bf F}(\bx_{\mi})=-\dfrac{\dif E}{\dif\bx_{\mi}}=-\dfrac{1}{2}q_{\mi}\dfrac{\dif\varphi(\bx_{\mi})}{\dif\bx_{\mi}}.
\label{eq:force_formula}
\end{align}
Applying this on \eqref{eq:ewald1p_real}-\eqref{eq:ewald1p_self} we obtain
\begin{align}
{\bf F}^{1\P}(\bx_{\mi})= &{\bf F}^\R(\bx_{\mi})+{\bf F}^\F(\bx_{\mi})+{\bf F}^{k_3=0}_{\mi} \nonumber \\
=& q_{\mi}\sum_{\bp\in P_1}^{'}\sum_{\ni=1}^Nq_\ni\left(\frac{2\xi}{\sqrt{\pi}}e^{-\xi^2|\bx_{\mi\ni,\bp}|^2}+\frac{\erfc(\xi|\bx_{\mi\ni,\bp}|)}{|\bx_{\mi\ni,\bp}|}\right)\frac{\bx_{\mi\ni,\bp}}{|\bx_{\mi\ni,\bp}|^2} \label{eq:force_real} \\
 & -\dfrac{\ii}{2\pi L_3}q_{\mi}\sum_{k_3\neq0}\sum_{\ni=1}^N q_\ni\int_{\mathbb{R}^2} \frac{\bk e^{-(\kappa^2+k_3^2)/4\xi^2}}{\kappa^2+k_3^2}e^{\ii\bk\cdot(\bx_{\mi}-\bx_{\ni})}\dif\bka \label{eq:force_fourier} \\
&+\dfrac{1}{L_3}q_{\mi}\sum_\subindex{\ni=1}{\ni\neq \mi}^Nq_{\ni}\dfrac{\br_{\mi\ni,0}}{\rho_{\mi\ni}^2}\left(1-e^{-\xi^2\rho_{\mi\ni}^2} \right), \label{eq:force_zero}
\end{align}
where 
\begin{align*}
\bx_{\mi\ni,\bp} &= \bx_{\mi}-\bx_{\ni}+\bp, \\ 
\br_{\mi\ni,0}&=(\br_{\mi}-\br_{\ni},0).
\end{align*}
The real space sum \eqref{eq:force_real} is then evaluated as in the 3d-periodic case. The $k_3=0$ term \eqref{eq:force_zero} can be again embedded into the nonzero Fourier sum with almost no cost. The integral in \eqref{eq:force_fourier} can be evaluated using the fact that it is a differentiation of the incomplete modified Bessel function in \eqref{eq:ewald1p_fourier_complete}. This approach is used to compute the force reference solution. 

Among different fast methods to evaluate the force, analytic differentiation of the potential has been shown to be the most efficient method which also preserves spectral accuracy, see \cite{Deserno2010}. Therefore, to evaluate the $k$-space sum, we do not proceed as in the formula in \eqref{eq:force_fourier} since it needs 2 extra 3D FFTs. Instead we differentiate \eqref{eq:se1p_complete_discretized} with respect to $\bx_{\mi}$. We have
\begin{align}
{\bf F}^\F(\bx_{\mi}) &= Ch^3\sum_{\ni}(\bx_{\mi}-\bx_{\ni})_{\ast,3}\wt{H}(\bx_{\ni})e^{-2\xi^2|\br_{\mi}-\br_{\ni}|^2/\eta}e^{-2\xi^2(z_{\mi}-z_{\ni})_{\ast}^2/\eta},
\label{eq:force_se1p}
\end{align}
where $C=\frac{2}{L_3}\frac{\xi^2}{\eta}\left(\frac{2\xi^2}{\pi\eta}\right)^{3/2}$ and $(\cdot)_{\ast,3}$ denotes that the periodicity is applied on the third dimension only. We remark that this sum and \eqref{eq:se1p_complete_discretized} can be computed concurrently.

%===================================================================================================================
%===================================================================================================================

%===================================================================================================================
%===================================================================================================================
\section{Approximation errors and parameter selection}
\label{sect:approxerr_parsel}
In addition to the truncation error due to the finite representation of the $k$-space sum (section \ref{sec:truncation_error}), \textit{approximation errors} are also involved in our fast method. These errors are committed due to (a) the truncation of Gaussians, (b) applying the quadrature rule to evaluate \eqref{eq:se1p_complete} and (c) approximating Fourier integrals. The approximation errors due to (a) and (b) are considered in section  \ref{subsec:approximation_error} and (c) is discussed in section \ref{sec:spectral_accuracy}. 
\subsection{Truncating and resolving Gaussians}
\label{subsec:approximation_error}

To construct our fast method in section \ref{subsec:foundation} we introduced a free parameter $\eta$ which can control the width of Gaussians. As pointed out in \cite{Lindbo2011} a suitable choice is to set 
\begin{align}
\eta=\left( \dfrac{2w\xi}{m} \right)^2,
\label{eq:eta}
\end{align}
where $w$ denotes the half width and $m$ the shape parameter of a Gaussian, see figure \ref{fig:support}.
\begin{figure}[htbp]
\centering 
\includegraphics[width=.4\textwidth,height=.3\textwidth]{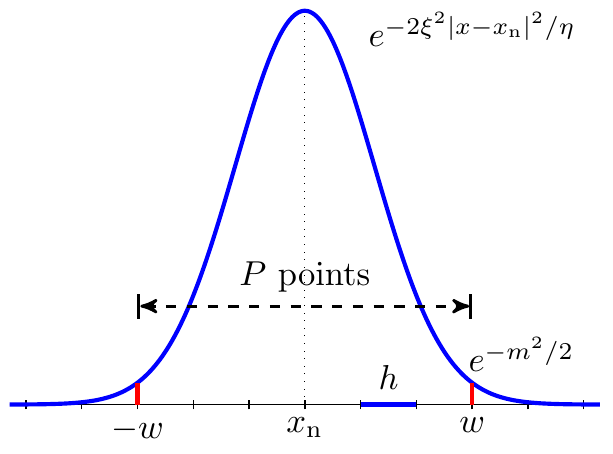}
\caption{Gaussians are truncated with $P$ points in the support on $[x_\ni-w,x_\ni+w]$.}
\label{fig:support}
\end{figure}
Let $P\leq M$ be the number of points in the support of each Gaussian in each direction and $h=L/M$ be the grid spacing. Therefore, $w=Ph/2$.

Using the parameters defined above, the following theorem provides an estimate for the approximation error due to (a) and (b).
\begin{theorem}
Given the input parameters $\xi>0$, $h>0$ and a positive odd integer $P\leq M$, let $w=Ph/2$ and define $\eta$ as in \eqref{eq:eta}. The error committed in evaluating equation \eqref{eq:ewald1p_fourier_integral} by truncating the Gaussians at $|\textsc{\bx}-\textsc{\bx}_{\textsc{\bn}}|=w$ and applying the trapezoidal rule $\varphi^\F$ \eqref{eq:se1p_complete_discretized} can be estimated by
\begin{align}
|\varphi^\F-\varphi_{_\mathrm{SE1P}}^\F|_{\infty}\leq C(e^{-\pi^2P^2/2m^2}+{\mathrm{erfc}}(m/\sqrt{2})).
\label{eq:approximation_error}
\end{align}
\end{theorem}
Balancing both terms, \eqref{eq:approximation_error} thereupon simplifies to read 
\begin{align}
|\varphi^\F-\varphi_{_\mathrm{SE1P}}^\F|_{\infty}\lesssim Ae^{-c^2\pi P/2},
\label{eq:approximation_error_oneterm}
\end{align}
in which we have used $m= c\sqrt{\pi P}$, $c<1$, see \cite{Lindbo2011}. This leaves us with a single parameter $P$ to control these types of errors. 

In figure \ref{fig:se1p_approx_error} we plot the scaled rms error in evaluating the $k$-space part of the potential as a function of $P$ together with the error estimate given in \eqref{eq:approximation_error_oneterm}. In this figure, the error is scaled with $A=\sqrt{Q\xi L}/L$. We run the simulation for 100 different systems with $N=100,200,300,400$, $L=1,5,10,20,40$ and $\xi=(5,15,25,30,35)/L$. This confirms the validity and sharpness of the approximation error estimate. In addition, in the absence of the other errors, $P\approx12$ and $\approx24$, are sufficient to achieve single and double precision accuracies respectively.
We numerically demonstrate that the approximation error in computing the force takes the form of $e^{-c^2\pi P/2}$ but with a different constant, $A=Q\sqrt{\xi L}/L$.
\begin{figure}[htbp]
  \centering \includegraphics[width=0.4\textwidth,height=.4\textwidth]{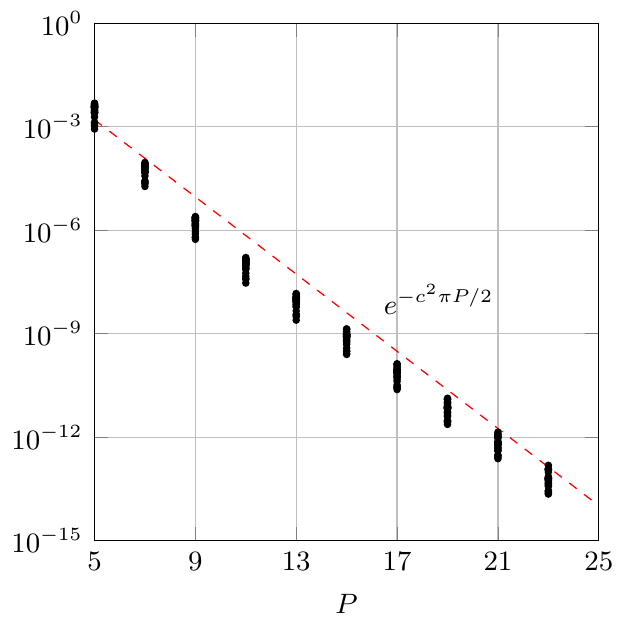}
  \caption{The estimate (red dashed line) and measured absolute rms error in evaluating the potential scaled with $A=\sqrt{Q\xi L}/L$ as a function of the number of points in the support of Gaussians $P$ for 100 different systems. We choose $N=100,200,300,400$, $L=1,5,10,20,40$ and $\xi=(5,15,25,30,35)/L$ and $c=0.95$. Other parameters such as grid size and oversampling factor are chosen appropriately that the other errors are negligible.}
  \label{fig:se1p_approx_error}
\end{figure}

%===================================================================================================================
%===================================================================================================================

%===================================================================================================================
%===================================================================================================================
\subsection{Upsampling for computing inverse Fourier transforms}
\label{sec:spectral_accuracy}
Approximation errors are introduced also due to approximating the Fourier integrals \eqref{eq:mft} and \eqref{eq:imft}.
In this section we wish to address the need for upsampling and relate the choice of upsampling factor $\sl$ and local pad size $\nl$ to approximation errors. We also show that the approximation of the Fourier integrals is spectrally accurate.

The Fourier integral in \eqref{eq:mft} has a fast decaying integrand. In \cite{Lindbo2012} the authors show that the integral can be computed without any upsampling up to the desired error tolerance. However, as we explained in section \ref{sec:discretization}, this is not the case for the inverse Fourier integral in \eqref{eq:imft} since the integrand varies quickly near $\kappa=0$ for small $|k_3|$. Therefore in the next theorem we consider an infinite Fourier integral similar to the one given in \eqref{eq:imft} in one dimension.
\begin{theorem}
\label{theo:spectral_accuracy}
Let $k_3,\alpha>0$ and define the integral 
\begin{align}
F=\int_{\mathbb{R}}\hat{f}(k)\dif k, \quad \hat{f}=\dfrac{e^{-\alpha(k^2+k_3^2)}}{k^2+k_3^2}.
\label{eq:fourier_integral_theorem}
\end{align}
For any $h>0$, define the trapezoidal rule approximation
\begin{align}
T_h = h\sum_{j=-\infty}^{\infty} \hat{f}(jh),\quad h>0,
\label{eq:theorem_trapezoidal}
\end{align}
then
\begin{align}
|T_h-F| = \dfrac{2\pi}{k_3}\dfrac{1}{e^{2\pi k_3/h}-1}.
\label{eq:proof_I}
\end{align}
\end{theorem}

\begin{proof} See \ref{appendix:proof}.
\end{proof}

The corresponding two-variable form of the integral in \eqref{eq:fourier_integral_theorem} can also be approximated using the trapezoidal rule. To obtain an error estimate for this approximation, first note that the integral in \eqref{eq:fourier_integral_theorem} evaluates as
\begin{align*}
F_1=\int_{\mathbb{R}}\dfrac{e^{-\alpha(k^2+k_3^2)}}{k^2+k_3^2}\dif k=\dfrac{\pi}{|k_3|}\erfc(\sqrt{\alpha}|k_3|),
\end{align*}
and the related two dimensional integral as
\begin{align*}
F_2=\int_{\mathbb{R}^2}\dfrac{e^{-\alpha(k_1^2+k_2^2+k_3^2)}}{k_1^2+k_2^2+k_3^2}\dif k_1\dif k_2=\pi \E(\alpha k_3^2),
\end{align*}
Observing that $\erfc(x)$ and $\E(x^2)$ behave similarly for large $x$, we obtain our heuristic error estimate in approximating $F_2$ using the trapezoidal rule,
\begin{align}
2\pi C e^{-2\pi |k_3|/h}.
\label{eq:proof_I2d}
\end{align}
The error estimate \eqref{eq:proof_I2d} supports our claim that for large enough $\nl $, there is no need for upsampling. In the following examples we assess the accuracy of the equality \eqref{eq:proof_I} and show the reliability of \eqref{eq:proof_I2d}. We let $\alpha=0.1$ and choose $C=1$. Figure \ref{fig:error_theorem} shows the measured error and the approximation error estimates \eqref{eq:proof_I} and \eqref{eq:proof_I2d} as a function of grid spacing $h$ for $k_3\in\lbrace1,2,3\rbrace$. As the error bound indicates, the approximation error decays as a function of $h$ and faster than any power of $h$. Moreover, we observe that for sufficiently large $|k_3|$, any level of accuracy can be achieved with $h=1$.

\begin{figure}[htbp]
  \centering
  \begin{subfigure}[b]{0.49\textwidth}
    \centering
    \includegraphics[width=\textwidth,height=.8\textwidth]{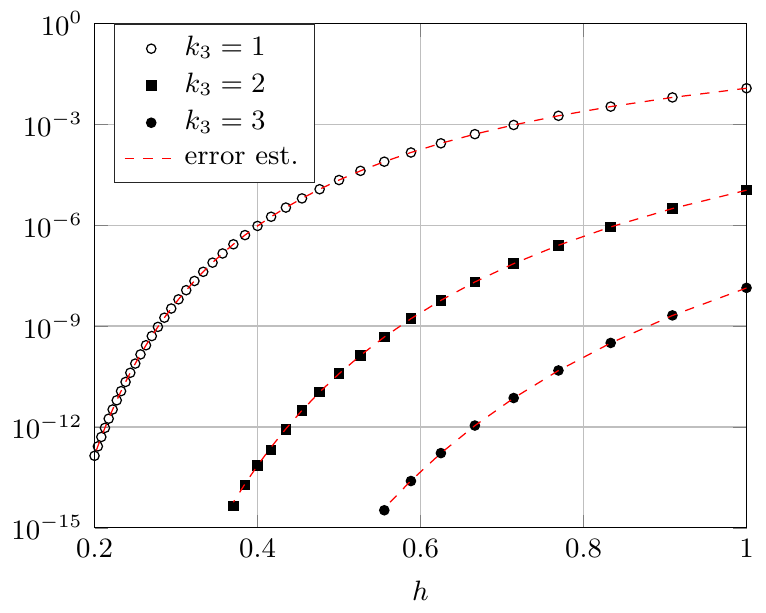}
  \end{subfigure}
  \begin{subfigure}[b]{0.49\textwidth}
    \centering 
    \includegraphics[width=\textwidth,height=.8\textwidth]{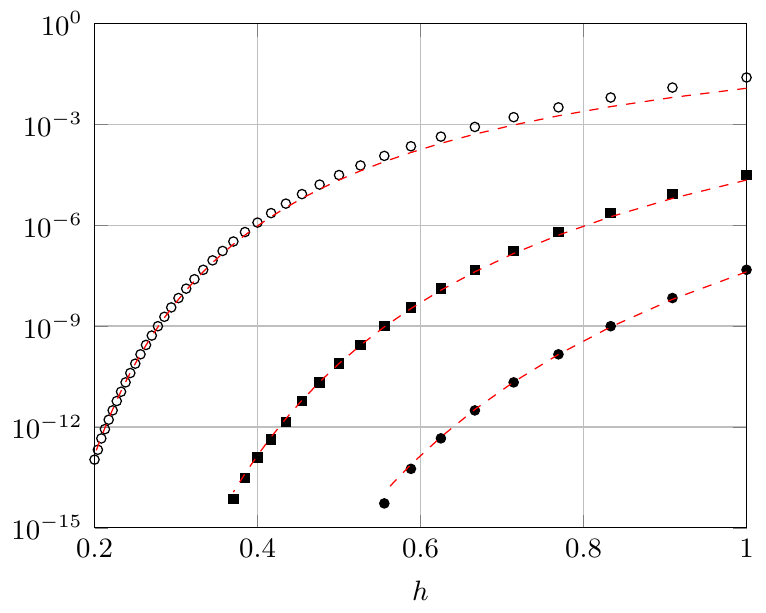}
  \end{subfigure}
  \caption{Measured error and approximation error estimate (dashed red) as a function of $h$ and for $k_3\in\lbrace1,2,3\rbrace$ with $\alpha=0.1$. (Left) For the one variable integral $F_1$ and using the error identity \eqref{eq:proof_I}. (Right) For the two-variable integral $F_2$ and using the error estimate \eqref{eq:proof_I2d} with $C=1$.}
  \label{fig:error_theorem}
\end{figure}

As has been stated before, the integrals examined here and the Fourier integral in \eqref{eq:imft} are similar but not the same. Therefore, the error estimate \eqref{eq:proof_I2d} does not give the actual error in approximating the Fourier integral \eqref{eq:imft} using the trapezoidal rule.  It however, gives an understanding of the exponential decay with $|k_3|$ and it can be used to compute the upsampling factor. To relate the actual error and the error estimate in \eqref{eq:proof_I2d}, we define $h=\frac{2\pi}{L}\frac{1}{\sl }$ and $k_3=\frac{2\pi \nl}{L}$. In section \ref{subsec:parameter_selection} we discuss how to choose $\sl $ and $\nl$ in practice. In figure \ref{fig:sgEqual1} we numerically demonstrate that if $\sl $ and $\nl$ are chosen properly, no oversampling is needed to resolve the quadrature of the Fourier integrals on $\mathbb{J}$ \eqref{eq:J}. Note that we always require that the simulation box is extended by $2w=Ph$ in each non periodic dimension to include the support of truncated Gaussians. 

\begin{figure}[htbp]
  \centering 
  \includegraphics[width=.6\textwidth]{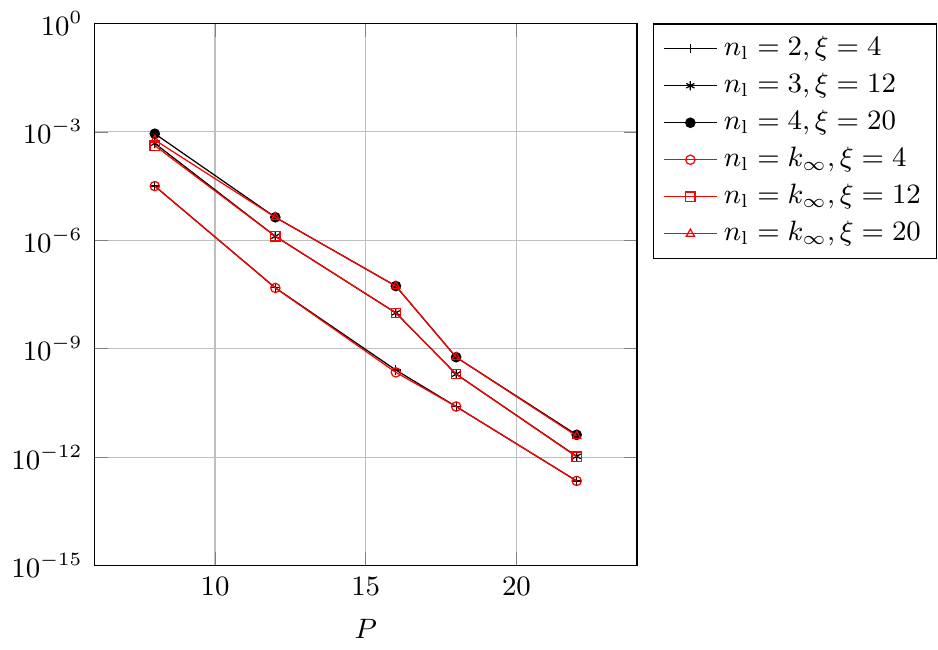}
  \caption{Error in the $k$-space sum vs. $P$ for a system with $L=1$, $N=100$ and $\xi\in\lbrace4,12,20\rbrace$. An upsampling is used to either upsample all modes ($\nl=k_\infty$) or $\nl$ modes with $\nl\in\lbrace2,3,4\rbrace$. $M$ is large enough such that the truncation error is negligible.}
  \label{fig:sgEqual1}
\end{figure}

%===================================================================================================================
%===================================================================================================================

%===================================================================================================================
%===================================================================================================================
\subsection{Parameter selection}
\label{subsec:parameter_selection}
There are several parameters involved in the computation of the Ewald sum and more parameters are included with the introduction of a fast method. The decomposition parameter $\xi$ is chosen such that the runtime of evaluating the real space and $k$-space sums are similar. This selection however, depends on the algorithm and implementation. With a given error tolerance and $\xi$ and by using \eqref{eq:rc} and \eqref{eq:kinf}, the cut-off radius $r_\c$ and the bound for the Fourier modes $k_{\infty}$ are computed. The number of grid points in the periodic direction can then be computed by $M=2k_{\infty}$.

In the SE1P method, we have introduced another parameter $P\leq M$ which controls the width of Gaussians. This parameter is computed by inverting \eqref{eq:approximation_error_oneterm} and is set to even integers. We also set the Gaussian shape parameter as $m=0.95\sqrt{\pi P}$. This gives that 
\begin{align*}
\eta=\left( \dfrac{2w\xi}{m} \right)^2=\dfrac{P\xi^2h^2}{c^2\pi},
\end{align*}
where $c=0.95$, cf.\ \eqref{eq:comp_eta}.

Moreover, we need to set the oversampling factors $\sl $, $s_0$ and the local pad size $\nl$. We need to choose $\sl $ such that the error in computing the most difficult mode $k_3=\frac{2\pi}{L_3}\cdot 1$ is less than the error tolerance. Therefore using \eqref{eq:proof_I2d} and an error tolerance $\e$, we require that
\begin{align}
e^{-2\pi \sl }<\e.
\label{eq:choose_sl}
\end{align}
We also observe that $\nl $ and the grid size $M$ are related and in practice we can use
\begin{align}
M\approx 10 \nl.
\label{eq:choose_nl}
\end{align}
In summary, the process of selecting parameters in our algorithm is as follows. For a given error tolerance $\e$ and the splitting parameter $\xi$, $r_\c$ and $k_\infty$ are computed using \eqref{eq:trunc_error_estimate_real} and \eqref{eq:trunc_error_estimate_fourier}. Using \eqref{eq:approximation_error_oneterm}, $P$ is selected. Upsampling parameters $\sl $ and $\nl$ are chosen via \eqref{eq:choose_sl} and \eqref{eq:choose_nl} and $\sl $ is adjusted such that $\sl \tilde{M}=\sl (M+P)$ is integer. Since $s_0$ oversampling factor has a negligible effect on the total runtime, its selection is rather simple. We choose $s_0\approx 2.4$ such that $s_0\tilde{M}$ is also integer. Moreover, as discussed in section \ref{subsec:kspace}, $R=\sqrt{2}\tilde{L}$.

In the following example we compute as a reference, all the input parameters for a sample system with two decomposition parameters and two error tolerances.

\begin{example}
Consider a system of $N=120$ uniformly distributed particles located in a box of size $L=10$ with $Q=1$ and $\xi\in\lbrace 1.5,3\rbrace$. In table \ref{tab:example} we list the input parameters required to achieve absolute rms error below set tolerances $10^{-6}$ and $10^{-12}$ in evaluation of the $k$-space potential. As a comparison, the FFT requires $\approx3M^3\log(M)$ operations in the 3d-periodic case and approximately
\begin{align*}
  M^3\log(M)+4\nl \sl ^2\tilde{M}^2\log(\sl \tilde{M})+2s_0^2\tilde{M}^2\log(s_0\tilde{M})
\end{align*}
operations in the 1d-periodic case, cf.\ Algorithm \ref{alg:afft}.

%For instance, in the case of $\xi=3$ and for the rms error tolerance of $10^{-6}$, the FFT needs $\approx 12\times100^3$ operations in 1d-periodic and $\approx4\times100^3$ operations in 3d-periodic case.
\begin{table}[htbp]
\small
\begin{center}
  \begin{tabular}{ccccccccc}
	$\xi$ & $M$& $P$ & $\tilde{M}$ & $\nl$ & $\sl $ & $\sl \tilde{M}$ & $s_0\tilde{M}$ & rms error\\ 
    \hline 
	1.5 & 32 & 12 & 44 & 3 & 2   & 88  & 106 & $5.4\times10^{-7}$\\
     3  & 64 & 12 & 76 & 6 & 2.4 & 184 & 184 & $9.4\times10^{-7}$\\
    1.5 & 52 & 24 & 76 & 7 & 3.2 & 244 & 184 & $6.7\times10^{-13}$\\
	 3  & 96 & 24 & 120& 14& 3.9 & 468 & 288 & $8.3\times10^{-13}$\\
  \end{tabular}
\end{center}
\caption{Input parameters for a sample system of $N=120$ uniformly distributed particles in a box of size $L=10$, with $Q=1$ and $\xi=\lbrace1.5,3\rbrace$ to achieve single and double precision accuracy. Also $s_0\approx2.4$.}
\label{tab:example}
\end{table}
\end{example}

%===================================================================================================================
%===================================================================================================================

%===================================================================================================================
%===================================================================================================================
\section{Implementation of the SE1P method}
\label{sec:implementation}
\subsection{AFT/AIFT}
We now present an algorithm which accelerates the computation of mixed Fourier transforms introduced in \cite{Lindbo2012}. As has been stated before, our system of interest is periodic in $z$ and free in the $x$ and $y$ directions. Consider again the local pad set $\mathbb{I}$ \eqref{eq:I}, and the associated set $\mathbb{J}$ \eqref{eq:J}, local pad size $\nl$ and oversampling factors $\sl $ and $s_0$ introduced in section \ref{subsec:kspace}. Moreover, let $M$ be the number of grid points and let $k_3\in\frac{2\pi}{L_3}\lbrace-\frac{M}{2},\ldots,\frac{M}{2}-1\rbrace$ and $\kappa_1,\kappa_2\in\frac{2\pi}{s\tilde{L}}\lbrace-\frac{\tilde{M}}{2},\ldots,\frac{\tilde{M}}{2}-1\rbrace$, with $s$ defined in \eqref{eq:upsamplings}. The AFT algorithm to compute $\wh{H}$ from $H$ has the following steps.

{
\small
\begin{breakablealgorithm}
\caption{Adaptive Fourier transform - AFT}
\small
\label{alg:afft}
\begin{algorithmic}[1]
\INPUT Grid-representation of sources $H$, cf.\ \eqref{eq:spread_1p}, grid size $M$, oversampling factors $s_0, \sl $, and $\mathbb{I}$, $\mathbb{J}$ sets, cf.\ \eqref{eq:I} and \eqref{eq:J}.
\State Apply a 1D FFT on $H$ in the $z$ direction to compute $\wh{H}(x,y,k_3)$ with $\order{M^3\log(M)}$ arithmetic operations.
\State Pad $\wh{H}(x,y,0)$ with zeros in the $x$ and $y$ directions with oversampling factor $s_0$ and apply a 2D FFT to compute $\wh{H}(\kappa_1,\kappa_2,0)$, with $\order{2\nl \sl ^2\tilde{M}^2\log(\sl ^2\tilde{M}^2)}$ arithmetic operations.
\State Pad $\wh{H}(x,y,\mathbb{I})$ with zeros in the $x$ and $y$ directions with oversampling factor $\sl $ and apply a 2D FFT for $k_3\in\mathbb{I}$ to compute $\wh{H}(\kappa_1,\kappa_2,\mathbb{I})$ with $\order{2s_0^2\tilde{M}^2\log(s_0^2\tilde{M}^2)}$ arithmetic operations.
\State Apply a 2D FFT on $\wh{H}(x,y,\mathbb{J})$ to compute $\wh{H}(\kappa_1,\kappa_2,\mathbb{J})$ with $\order{M(M-2\nl -1)^2\log(M)}$ arithmetic operations.
\OUTPUT Adaptive Fourier transform of $H$, $\wh{H}$, cf.\ \eqref{eq:mft_trap}.
\end{algorithmic}
\end{breakablealgorithm}
}

Now, assume that $\wh{H}(\kappa_1,\kappa_2,0)$, $\wh{H}(\kappa_1,\kappa_2,\mathbb{I})$ and $\wh{H}(\kappa_1,\kappa_2,\mathbb{J})$, by virtue of applying the AFT algorithm, exist. The AIFT algorithm to compute $H$ from $\wh{H}$ has the following steps.
{
\small
\begin{breakablealgorithm}
\caption{Adaptive inverse Fourier transform - AIFT}
\small
\begin{algorithmic}[1]
\INPUT Scaled adaptive Fourier transformed Grid-representation of sources $\wh{\wt{H}}$, cf.\ \eqref{eq:scale_1p}, grid size $M$, oversampling factors $s_0, \sl $, and $\mathbb{I},\mathbb{J}$ sets  cf.\ \eqref{eq:I} and \eqref{eq:J}.
\State Apply a 2D IFFT on $\wh{H}(k_1,k_2,\mathbb{J})$ in the $x$ and $y$ directions to compute $\wh{H}(x,y,\mathbb{J})$.
\State Apply a 2D IFFT on $\wh{H}(k_1,k_2,\mathbb{I})$ in the $x$ and $y$ directions, restrict the solution to $M$ grid points in each direction to compute $\wh{H}(x,y,\mathbb{I})$.
\State Apply a 2D IFFT on $\wh{H}(k_1,k_2,0)$ in the $x$ and $y$ directions, restrict the solution to $M$ grid points in each direction to compute $\wh{H}(x,y,0)$.
\State Merge $\wh{H}(x,y,\mathbb{I})$, $\wh{H}(x,y,\mathbb{J})$ and $\wh{H}(x,y,0)$ to construct $\wh{H}(x,y,k_3)$.
\State Apply a 1D FFT on $\wh{H}(x,y,k_3)$ in the $z$ direction to compute $H(x,y,z)$.
\OUTPUT Adaptive inverse Fourier transform of $\wh{\wt{H}}$, $\wt{H}$, cf.\ \eqref{eq:imft_trap}.
\end{algorithmic}
\end{breakablealgorithm}
}

The schematic representation of the AFT/AIFT, in two dimensions is shown in figure \ref{fig:oversampling}. The extension of the original box can be interpreted as zero padding the box or the memory required to fit the oversampled Fourier transformed charge distributions. 
\begin{remark}
The results of the steps 2-4 in algorithm \ref{alg:afft} can be stored and scaled separately. This is different from the illustration in figure \ref{fig:oversampling}.
\end{remark}
\begin{figure}[htbp]
\centering \includegraphics[width=.5\textwidth]{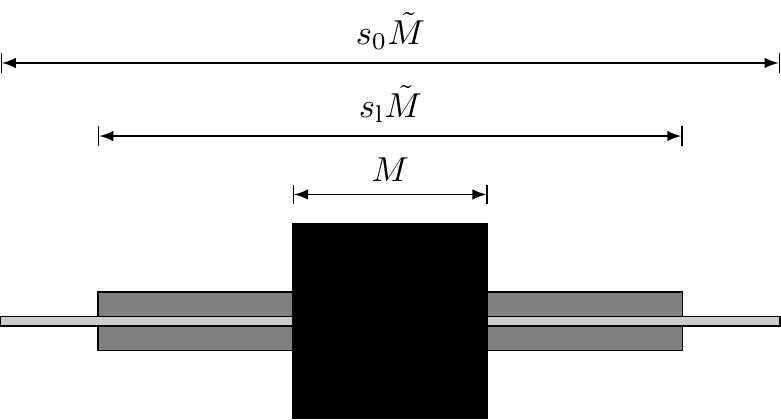}
\caption{Schematic representation of mixed Fourier Transforms in 2 dimensions.}
\label{fig:oversampling}
\end{figure}

To better illustrate the speed-up in the evaluation of FFTs, in table \ref{tab:mfft_mifft} we listed the runtime ratio between the computation of FFT/IFFT with $\sl =s_0$ and $\nl =k_\infty$, which we will refer to as \textit{plain upsampled FFT}, and AFT/AIFT for the cases $\sl \in\lbrace 2,4\rbrace$ and $\nl =8$. Using $\sl =2$, the ratio converges to approximately 2 and stays almost constant with respect to the grid size. For the case $\sl =4$ however, the ratio of the new algorithm to the plain upsampled FFT decreases to nearly 1:6. The role of adaptive Fourier transforms in lowering the cost of evaluating the FFT/IFFTs is more pronounced whenever the accuracy need, system size or Ewald parameter increase. Note that the AFT/AIFT algorithm has no effect on the gridding and gathering steps but it affects the runtime of the scaling step. 

\begin{table}[htbp]
\small
  \begin{center}
    \begin{tabular}{c|cccccccc}
      $M$ & 16&32&48&64&80&96&112&128\\
      \hline
      $\sl =2$ &0.78&1.43&1.77&2.23&2.28&2.51&2.64&2.51\\
      $\sl =4$ &1.23&2.43&3.36&4.65&5.06&5.99&6.85&6.30
    \end{tabular}
  \end{center}
  \caption{The FFT runtime ratio between the plain upsampled FFT and our new AFT algorithm. We used $\sl \in\lbrace 2,4\rbrace$ and $\nl =8$ for different grid sizes.}
  \label{tab:mfft_mifft}
\end{table}

%===================================================================================================================
%===================================================================================================================

%===================================================================================================================
%===================================================================================================================
\subsection{Fast Gaussian Gridding}
\label{subsec:fgg}
To increase the efficiency in computation of the exponential functions present in the gridding \eqref{eq:spread_1p} and gathering \eqref{eq:se1p_complete} steps, one can consider precomputing and reusing exponential functions. Assuming $M^3$ grid points and $N$ source points, each expression includes $NM^3$ number of $\exp(\cdot)$ evaluation which then can be reduced to $NP^3$ when the Gaussians are truncated to $P^3$ points in the support. But this is still expensive for large $N$ and high accuracy demands. Consider the evaluation of the expression 
\begin{align*}
e^{-\alpha(\bx-\bx_{\ni})^2}=e^{-\alpha(x-x_{\ni})^2}e^{-\alpha(y-y_{\ni})^2}e^{-\alpha(z-z_{\ni})^2},
\end{align*}
where $\bx_{\ni}$ is a source point and $\bx$ is located on a uniform grid. For simplicity we explain the procedure for the first term above. Since the expression is evaluated on a equispaced grid $x=ih$, $i=0,\ldots,(M-1)h$. We have
\begin{align*}
e^{-\alpha(ih-x_{\ni})^2}=f_1(i)\left[f_2(x_{\ni})\right]^if_3(x_{\ni})=e^{-\alpha(ih)^2}\left[e^{2\alpha h x_{\ni}}\right]^ie^{-\alpha x_{\ni}^2}.
\end{align*}
Clearly $f_1(i)$ is independent of $x_{\ni}$ and therefore can be computed, stored and reused. This operation, considering the evaluation in 3 dimensions, involves $M^3$ or after truncation $P^3$ evaluations of $\exp(\cdot)$ function. On the other hand, for each $x_{\ni}$, $f_2(x_{\ni})$ and $f_3(x_{\ni})$ are evaluated once and stored. Note that powers of $f_2$ can be simply computed by consecutive multiplication of the base $f_2(x_{\ni})$ by itself. These operations include $2N$ exponential evaluation and $P^3N$ multiplications. For the details of the implementation see \cite{Lindbo2011} and references therein. Hence, the SE1P method has a complexity of $\order{P^3N}+\order{M^3\log(M^3)}$. If we also account for the complexity of the real space sum, the grid size $M$ is tied to $N$ in the following way. If the simulation box size $L$ increases while the particle density $\frac{L^3}{N}$ is fixed, then $L\propto N^{1/3}$. Provided that $r_\textrm{c}$ and $\xi$ are kept fixed, the real space sum scales as $\order{N}$. Moreover, $M\propto L\propto N^{1/3}$ and therefore, $\order{M^3\log(M^3)}\propto\order{N\log(N)}$.

%===================================================================================================================
%===================================================================================================================

%===================================================================================================================
%===================================================================================================================
\section{Numerical results}
\label{sec:results}

In this section we present numerical results of computing the electrostatic potential and force with 1d-periodicity using the SE1P method. All the simulations are done on one core on a machine with Intel Core i7-3770 CPU which runs on 3.40 GHz with 8 GB of memory. The FFT/IFFT and scaling steps are done in MATLAB and the gridding and interpolation steps are written in C and are dynamically linked and called through the MATLAB MEX interface. The subroutines are written in C and are built with the GNU C Compiler at version 4.8.4. Our implementation is publicly available at {\cite{se_github}}. The package is accelerated with SIMD intrinsics and can be executed using OpenMP APIs. Moreover, the implementation allows for simulation of systems with non-cubic box shapes.

In example \ref{ex:wall}, error is measured using the absolute rms error defined in section \ref{sec:truncation_error}. In the other examples, we measure the relative rms error. In section {\ref{subsec:approximation_error}} we have presented formulas for the absolute errors, however we can approximate the magnitude of the potential and force (cf. figure {\ref{fig:se1p_approx_error}}), and hence can obtain estimates for relative errors. Therefore, parameters $M$ and $P$ can be computed using the absolute error formulas and approximate magnitude of the solution. Moreover, $s_0$ is kept fixed and $\nl $ and $\sl $ are optimized for set error tolerances in each example.

In the first example, (example \ref{ex:total}), we will consider the computation of the full potential. However, since the real space component stays essentially the same while changing the periodicity of the system, in the other examples we shall only consider the evaluation of the Fourier space part. 

%===================================================================================================================
\begin{example}
\label{ex:total}
We present here the total runtime of evaluating an approximation to the electrostatic potential {\eqref{eq:full_potential}} with 1d-periodicity for different system sizes $N$ and with a relative rms error of $\approx 2\times10^{-6}$. The real space component is computed using cell lists and the Fourier space part with the presented algorithm. In this example, we scale up the system such that the particle density $\frac{N}{L^3}$ stays constant. To choose $\xi$, a simple approach is to balance the runtime of the real space and Fourier space components of the Ewald sum for a moderately large system. With this approach we find $\xi=3.5$ to be an almost optimal value. Now one can follow the recipe given in section \ref{subsec:parameter_selection} to obtain the other parameters. We find $P=16$, $r_\c=0.9$, and $\sl =3$. Also $s_0\approx2.4$. Note that $P$, $s_0$ and $\sl $ are fixed since they depend only on the accuracy and not the system size. To keep the number of near neighbors fixed in evaluating the real space component, the cut-off radius $r_\c$ is kept constant. Considering a fixed $\xi$, the parameters $M$ and $\nl $ are functions of both accuracy and $L$. The total runtime of computing the potential and parameters to scale up the system are given in figure \ref{fig:real_fourier}.

\vspace{.5cm}
\begin{minipage}{0.49\textwidth}
  \centering 
  \includegraphics[width=1\textwidth,height=.8\textwidth]{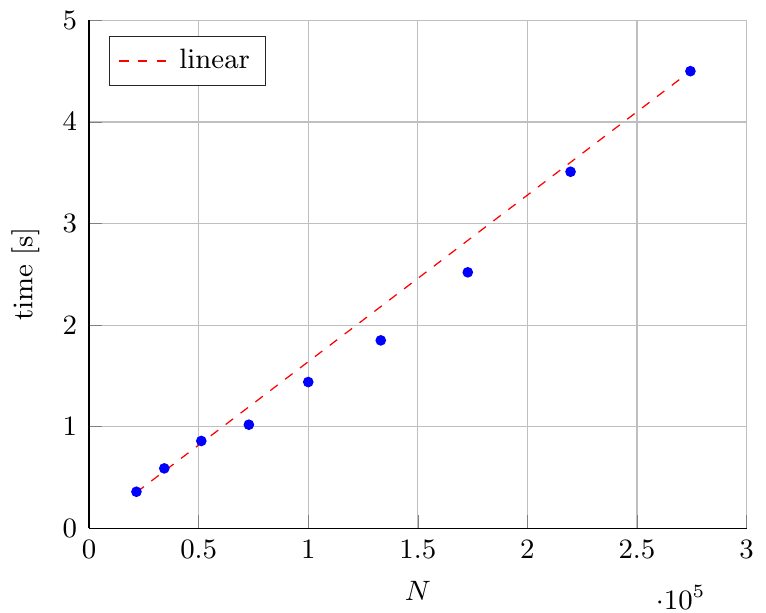}
\end{minipage}
\begin{minipage}{0.49\textwidth}
  \small
  \centering
  \begin{tabular}{rcccc}
    $N~~$& $L$ & $M$ & $\nl $\\
    \hline
    21\,600  & 6 & 46 & 4\\
    34\,300  & 7 & 54 & 5\\
    51\,200  & 8 & 62 & 6\\
    72\,900  & 9 & 64 & 6\\
    100\,000 & 10 & 80 & 6\\
    133\,100 & 11 & 84 & 7\\
    172\,800 & 12 & 92 & 8\\
    219\,700 & 13 & 100 & 9\\
    274\,400 & 14 & 108 & 10\\
  \end{tabular}
\end{minipage}
\captionof{figure}{(Left) Total runtime (real and Fourier space) in computing \eqref{eq:se1p_complete_discretized} using the SE1P algorithm. (Right) Parameters used to obtain the figure. Also $\sl =3$, $s_0\approx2.4$, $r_\c=0.9$, and $P=16$.}
  \label{fig:real_fourier}
\end{example}

%===================================================================================================================
\begin{example}
\label{ex:time_particle}
In this example, our aim is to study the behavior of different parts of the algorithm. We generate random systems of uniformly distributed particles with constant density $\frac{N}{L^3}=125$. This system is slightly more dense than the system in the previous example. We start with $N=20\,000$ and $L=5.43$ and scale up the system such that the number density remains constant. We set $\xi=4$ and plot the  runtime (figure \ref{fig:time_particle} (left)) and per-particle runtime (figure \ref{fig:time_particle} (right)) of computing the $k$-space component of the potential as a function of number of particles to achieve relative rms errors less than $10^{-5}$ and $10^{-9}$. Referring back to section \ref{subsec:fgg}, the computational complexity of the gridding and interpolation steps is of order $P^3N$ and FFT/IFFT is of order $N\log{N}$. Since the computations in this example are dominated by the gridding and interpolation steps, the time-particle plots for both tolerances scale linearly.

In figure \ref{fig:time_particle_detail}, we plot the runtime and per-particle runtime of the same systems in detail presenting the FFT/IFFT and scaling steps. The gridding and interpolation steps are not shown in this figure for clarity reasons and since their behaviors are similar to the total runtime in figure \ref{fig:time_particle}. The parameters $M$ and $P$ are obtained from \eqref{eq:kinf} and \eqref{eq:approximation_error_oneterm} respectively, $\sl =2$ for $\e=10^{-5}$ and $\sl =3$ for $\e=10^{-9}$. The other parameters used in this example are listed in table \ref{tab:params}.
\end{example}

\begin{figure}[htbp]
  \begin{subfigure}[b]{0.49\textwidth}
    \centering 
    \includegraphics[width=\textwidth,height=.8\textwidth]{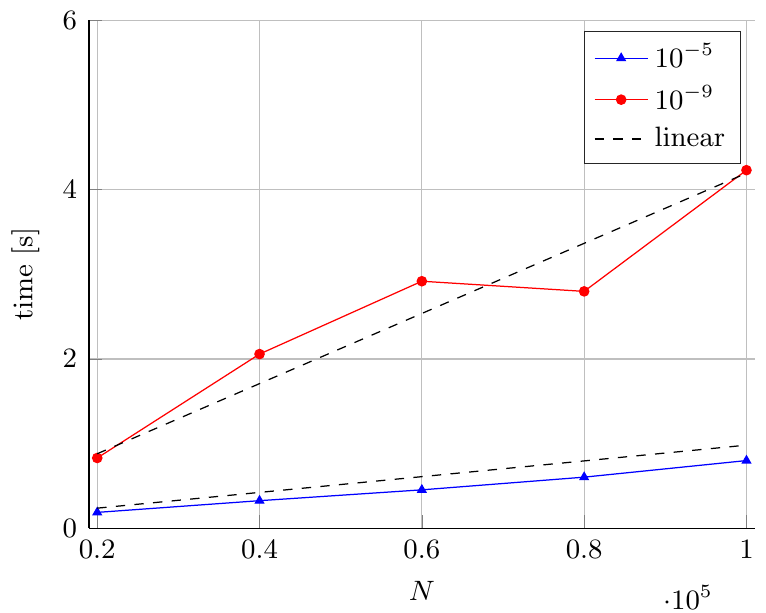}
  \end{subfigure}
  \begin{subfigure}[b]{0.49\textwidth}
    \centering 
    \includegraphics[width=\textwidth,height=.82\textwidth]{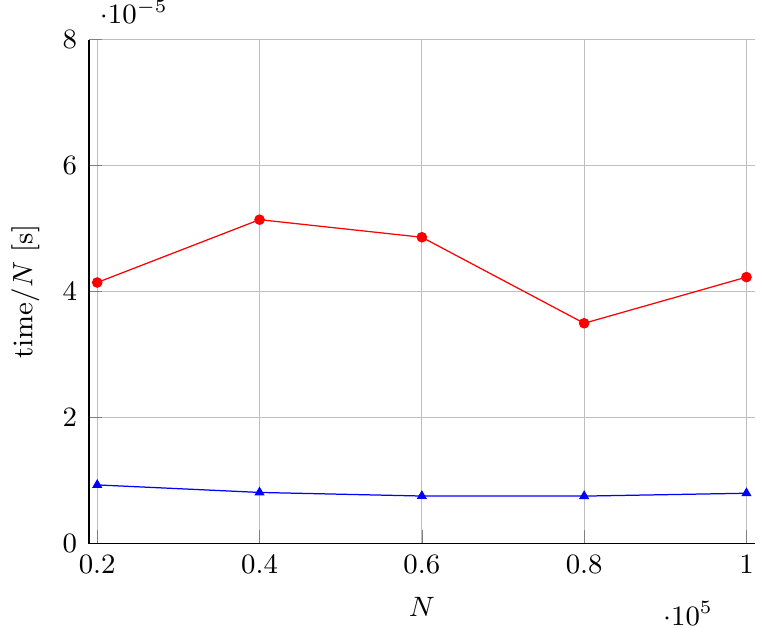}
  \end{subfigure}
  \caption{The $k$-space component runtime (left) and per-particle runtime (right) vs.\ number of particles to achieve relative rms errors less than $10^{-5}$ (blue) and $10^{-9}$ (red). Randomly generated systems with number density $\frac{N}{L^3}=125$ and $\xi=4$ are used. The parameters used in this example are listed in table \ref{tab:params}.}
  \label{fig:time_particle}
\end{figure}

\begin{figure}[htbp]
  \begin{subfigure}[b]{0.49\textwidth}
    \centering 
    \includegraphics[width=\textwidth,height=.8\textwidth]{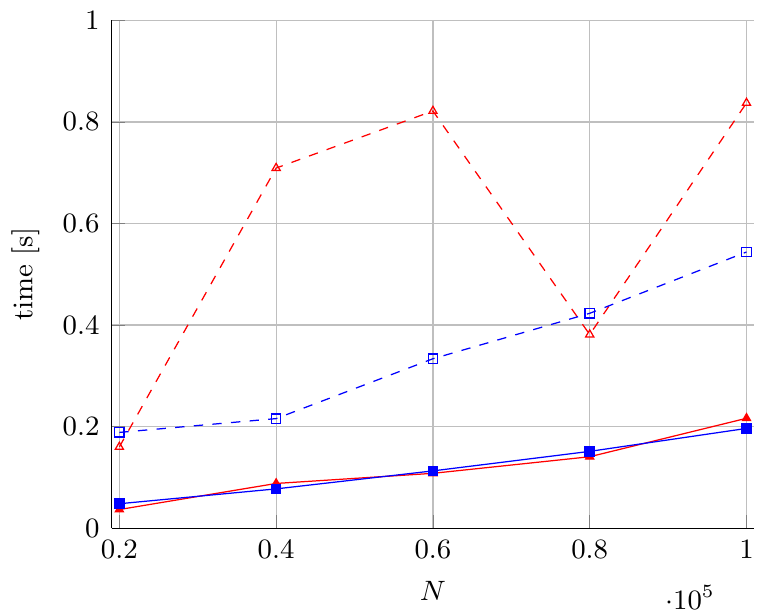}
  \end{subfigure}
  \begin{subfigure}[b]{0.49\textwidth}
    \centering 
    \includegraphics[width=\textwidth,height=.82\textwidth]{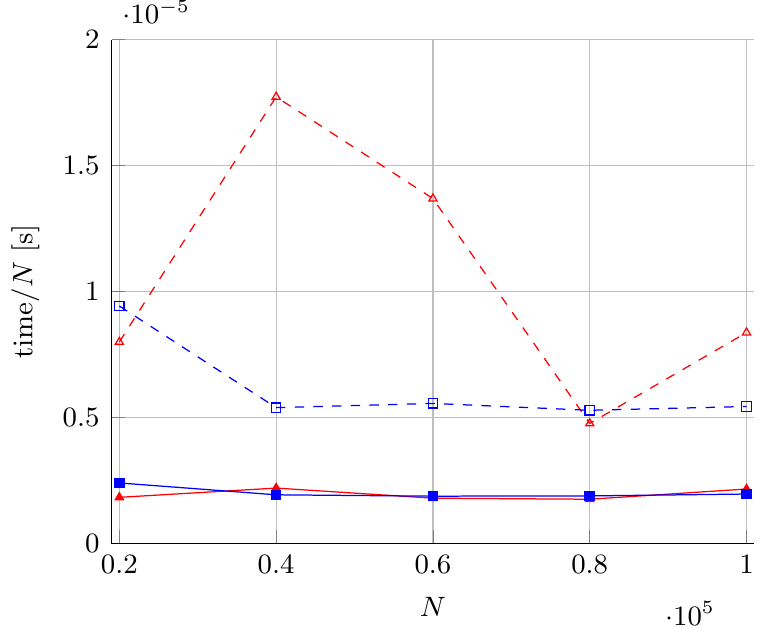}
  \end{subfigure}
  \caption{Runtime (left) and per-particle runtime (right) of FFT/IFFT ($\red\triangle$) and scaling ($\blue\square$) steps vs.\ number of particles to achieve relative rms errors less than $10^{-5}$ (solid) and $10^{-9}$ (dashed) respectively. Randomly generated systems with number density $\frac{N}{L^3}=125$ and $\xi=4$ are used. The parameters used in this example are listed in table \ref{tab:params}.}
  \label{fig:time_particle_detail}
\end{figure}

\begin{table}[htbp]
%\vspace{-.5cm}
\small
\begin{center}
\begin{tabular}{cccccc}
\hline 
\multirow{2}{*}{$N$} & \multirow{2}{*}{$L$} & \multicolumn{2}{c}{$\e=10^{-5}$} & \multicolumn{2}{c}{$\e=10^{-9}$} \\\cline{3-6}
& & $M$ & $\nl $ & $M$ & $\nl $ \\
\hline
20\,000 & 5.43 & 52 & 7&  68  & 10\\ 
40\,000 & 6.84 & 64 & 9&  82  & 11\\ 
60\,000 & 7.83 & 76 & 10& 94  & 13\\ 
80\,000 & 8.62 & 88 & 10& 104 & 14\\ 
100\,000 & 9.29& 96 & 12& 114 & 15\\
\end{tabular}
\end{center}
\caption{List of the parameters for examples \ref{ex:time_particle} and \ref{ex:1p3p} to obtain figures \ref{fig:time_particle}-\ref{fig:time_particle_detail} and \ref{fig:1p3p_tole5}-\ref{fig:1p3p_tole10_scale}. For tolerances $10^{-5}$ and $10^{-9}$ and  $\xi=4$, the cut-off radius is $r_\c=0.781$ and $1.08$, $\sl =2,3$ and $P=12,24$ respectively. $L$ is computed such that $\frac{N}{L^3}=125$ is satisfied.}
\label{tab:params}
\end{table}

%===================================================================================================================
\begin{example}
\label{ex:spectral}
In this example we consider a uniform system of $N=10$ particles with $L=1$ and we set $\xi=8$. We compute the force \eqref{eq:force_se1p} using the SE1P method. The grid size $M$ is chosen large enough such that the truncation error of the $k$-space sum is negligible. In figure \ref{fig:se1p_force} (left) we plot the relative rms error in computing the force as a function of $P$ for different oversampling factors. The figure shows that with $\sl =4$,  sufficiently large $\nl $ (5 in this case) and $P=24$ machine precision accuracy can be achieved. For $\sl =1$, no oversampling is made and therefore there is no need to choose $\nl $. In figure \ref{fig:se1p_force} (right) we show how the error decreases quickly as $\nl $ increases. In this figure the system and parameters are the same as in figure \ref{fig:se1p_force} (left) and $\sl =4$ is fixed. As we explained before, for the case of $\sl =1$ we still need to extend the computational domain to accommodate the support of the truncated Gaussians. This extension yields a bit of oversampling and therefore low error tolerances can still be achieved also with $\sl =1$.

\end{example}
\begin{figure}[htbp]
  \begin{subfigure}[b]{0.5\textwidth}
    \centering  \includegraphics[width=\textwidth,height=.8\textwidth]{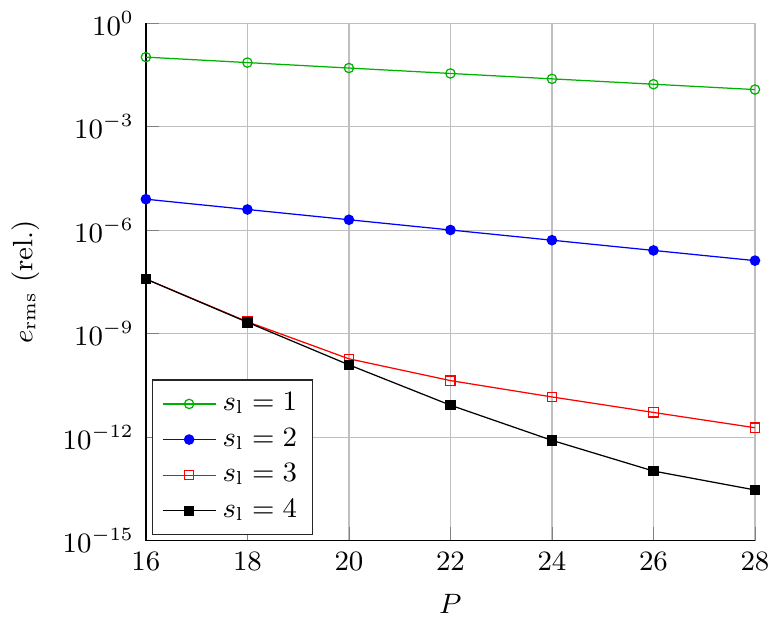}
\end{subfigure}
  \begin{subfigure}[b]{0.5\textwidth}
    \centering  \includegraphics[width=\textwidth,height=.8\textwidth]{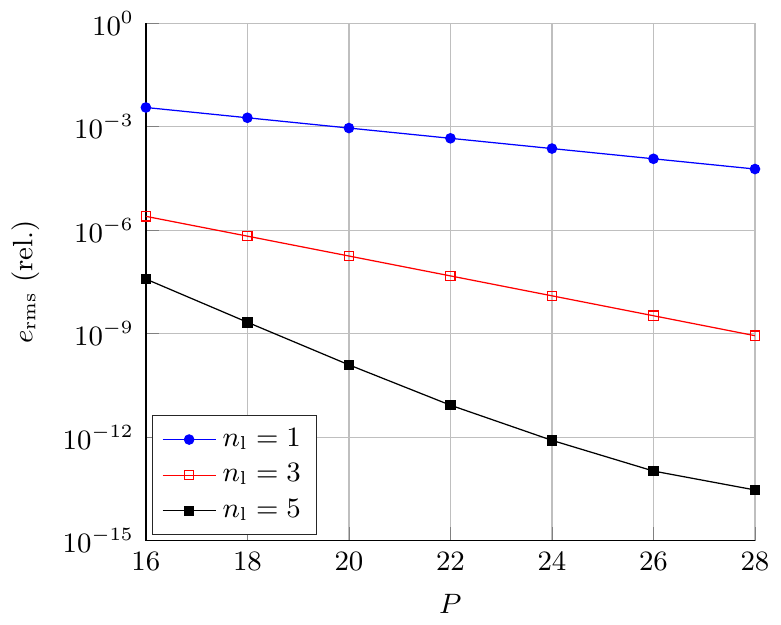}
\end{subfigure}
\caption{The rms error in computation of the force as a function of $P$ using the 1d-periodic spectral Ewald method. A uniform system of $N=10$ particles and $L=1$ with $\xi=8$ and $M=38$ is used. (Left) $\sl \in\lbrace1,2,3,4\rbrace$ and $\nl =6$. (Right) $\sl =4$ and $\nl \in\lbrace1,3,5\rbrace$.}
\label{fig:se1p_force}
\end{figure}

%===================================================================================================================
%===================================================================================================================
\begin{example}{\it (Finite size effect)}
\label{ex:finite_size}
Due to the inhomogeneity of the system in the non-periodic directions, a finite size effect may occur in evaluation of the $k$-space sum. This effect can be well quantified via different test cases that are well known to study such behavior (see \cite{{Lindbo2012},{DeJoannis2002d},{Widmann1997}}).
\subsubsection*{Case 1.}
We consider a system of $N=500$ oppositely charged particles and $L=1$ and evaluate the point-wise error of the force $k$-space sum for different oversampling factors $\sl $ as a function of the distance from the nearest box edge in $xy$-plane. We set $\xi=8$, $M=32$, $P=28$ and $\nl =4$. Figure \ref{fig:test_xy_accuracy} (left) shows that the point-wise error does not depend on the particle distance from the box edges. The error is computed as $\frac{1}{3}\sqrt{e_{x_1}^2+e_{x_2}^2+e_{x_3}^2}$, where $e_{x_i}$ is the point-wise absolute error in the $x$-direction.
\subsubsection*{Case 2.}
We use a system of $N=2$ oppositely charged particles with charges $q_i=\pm1$ and $L=1$. We let one of the particles be fixed at $(\frac{L_1}{2},\frac{L_2}{2},0)$ and the other particle move along the diagonal of the $xy$-plane, i.e., $x=y$, with $z=0.1$. We choose $L=1$, $\xi=8$, $M=32$, $P=28$, and $\nl =4$. Figure \ref{fig:test_xy_accuracy} (right) shows that the point-wise error in computing the force $k$-space sum for the free particle increases slightly as it gets far from the other particle. The error is computed as in the previous case.
\begin{figure}[htbp]
  \centering 
  \begin{subfigure}[b]{0.49\textwidth}
    \centering 
    \includegraphics[width=\textwidth,height=.8\textwidth]{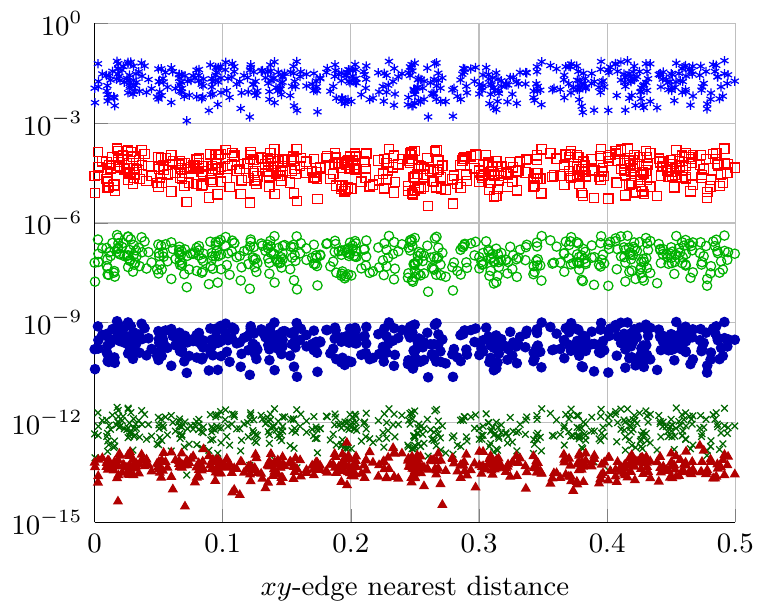}
\end{subfigure}
\hfill
  \begin{subfigure}[b]{0.49\textwidth}
    \centering 
    \includegraphics[width=\textwidth,height=.8\textwidth]{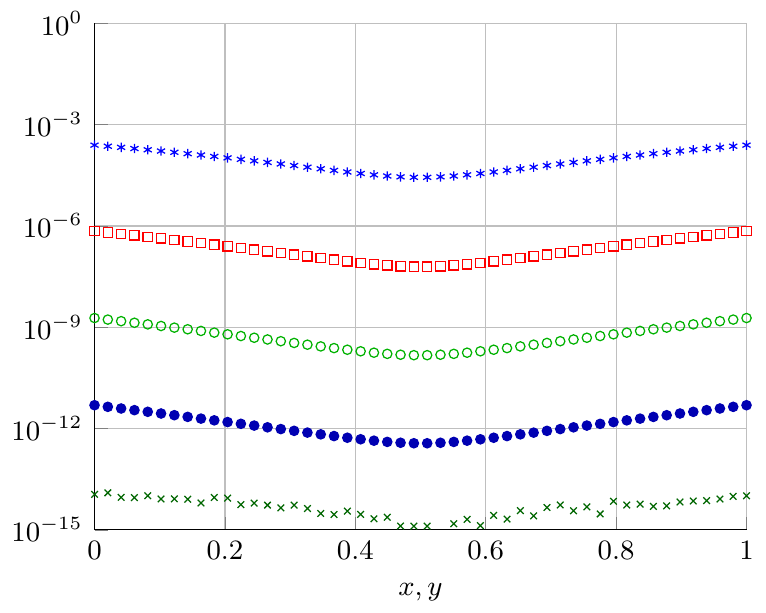}
\end{subfigure}
\caption{(Left) Point-wise error in the computation of the force $k$-space sum does not depend on the $xy$-edge nearest distance. (from top to bottom) $\sl \in\lbrace1,1.5,2,2.5,3,3.5\rbrace$. Also $N=500$. (Right) Point-wise error of the free particle in the two-particle system shows a small growth as it departs from the fixed one. (from top to bottom) $\sl \in\lbrace1,1.5,2,2.5,3\rbrace$. The grid size and $P$ are large enough such that the truncation error and approximation error of the Gaussians are negligible. In both plots we choose $L=1$, $\xi=8$, $M=32$, $P=28$ and $\nl =4$.}
\label{fig:test_xy_accuracy}
\end{figure}

\end{example}

%===================================================================================================================
\begin{example}{\it (Efficiency compared to the triply periodic case)}
\label{ex:1p3p}
In this example we compare the behavior of the SE1P and SE3P methods \cite{Lindbo2011} to achieve relative rms errors less than $10^{-5}$ and $10^{-9}$. We use the same systems and parameters as in example \ref{ex:time_particle} and compare only the $k$-space part of the potential. Figure \ref{fig:1p3p_tole5} (right) shows the $k$-space runtime comparison of the SE1P and SE3P methods to achieve an error less than $10^{-5}$. Evidently, the runtime ratio of algorithms remains approximately constant around 1.5 while we increase the system size. In figure \ref{fig:1p3p_tole5} (left) we plot the FFT/IFFT and scaling steps runtime comparison for both methods. The fluctuations in the FFT/IFFT curves are due to the fact that the FFT routine is more efficient for some grid sizes. In figures \ref{fig:1p3p_tol5_scale} (left and right) we plot the total and detailed per-particle runtime of the same systems and at a relative error level below $10^{-5}$. As figures suggest, for larger systems, both methods are efficient, i.e., per-particle runtime stays almost constant as the system size grows.

A similar experiment has been conducted for the error tolerance $10^{-9}$ and the results are shown in figures \ref{fig:1p3p_tole10} and \ref{fig:1p3p_tole10_scale}. Again the results show that the runtime ratio of the 1d- to 3d-periodic cases stays almost constant. To achieve this accuracy, $\sl =3$ and at most $2\nl +1=31$ Fourier modes are oversampled.
\begin{figure}[htbp]
  \begin{subfigure}[b]{0.49\textwidth}
    \centering 
    \includegraphics[width=\textwidth,height=.8\textwidth]{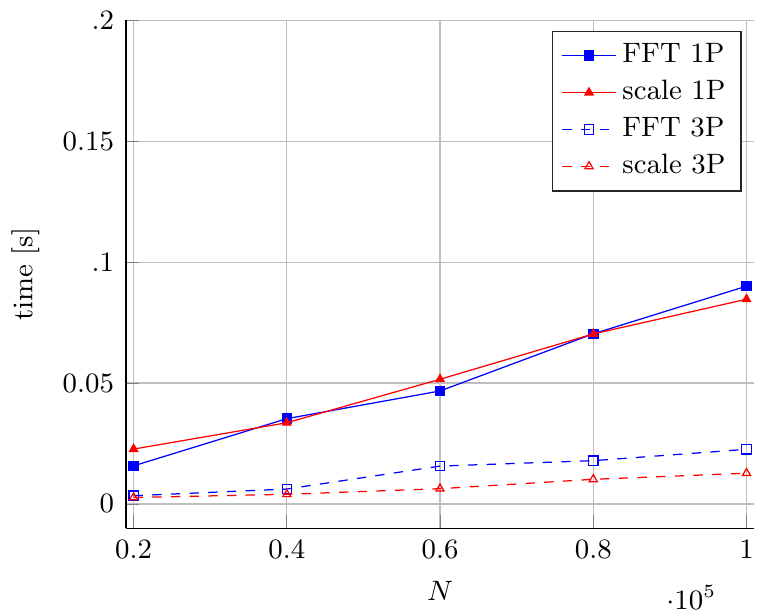}
  \end{subfigure}
  \begin{subfigure}[b]{0.49\textwidth}
    \centering 
    \includegraphics[width=\textwidth,height=.8\textwidth]{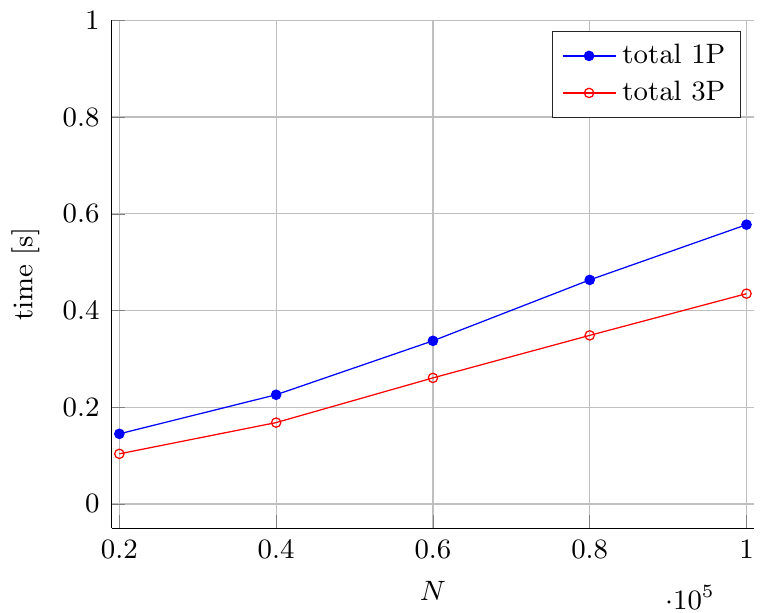}
  \end{subfigure}
  \caption{Runtime comparison of 1d- and 3d-periodic $k$-space sum vs. number of particles to achieve a relative rms error below $10^{-5}$. (Left) FFT/IFFT and scaling steps runtime. (Right) total runtime. We used $\xi=4$. The parameters used in this example are listed in table \ref{tab:params}.}
  \label{fig:1p3p_tole5}
\end{figure}
\begin{figure}[htbp]
  \begin{subfigure}[b]{0.49\textwidth}
    \centering 
    \includegraphics[width=\textwidth,height=.8\textwidth]{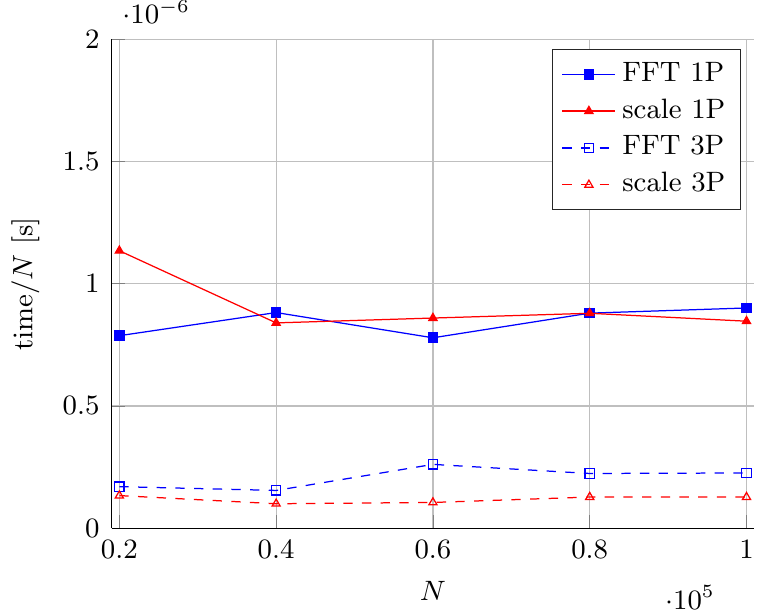}
\end{subfigure}
  \begin{subfigure}[b]{0.49\textwidth}
    \centering 
\includegraphics[width=\textwidth,height=.8\textwidth]{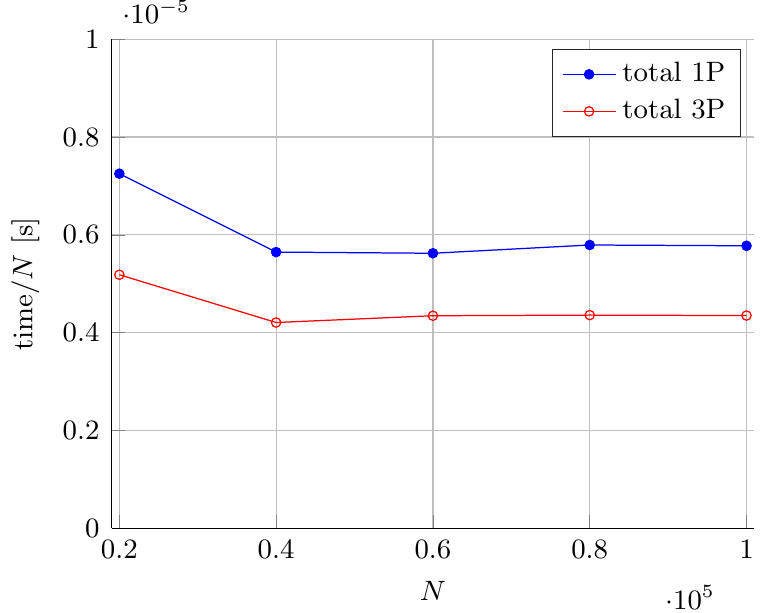}
\end{subfigure}
\caption{Per-particle runtime comparison of 1d- and 3d-periodic $k$-space sum vs. number of particles to achieve a relative rms error below $10^{-5}$. (Left) FFT/IFFT and scaling runtime. (Right) total runtime. We used $\xi=4$. The parameters used in this example are listed in table \ref{tab:params}.}
\label{fig:1p3p_tol5_scale}
\end{figure}
\begin{figure}[htbp]
  \begin{subfigure}[b]{0.49\textwidth}
    \centering 
    \includegraphics[width=\textwidth,height=.8\textwidth]{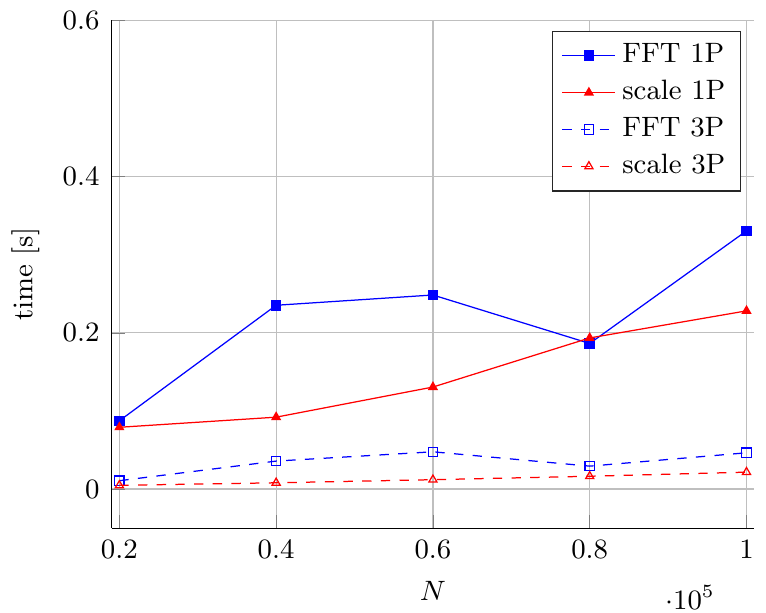}
  \end{subfigure}
  \begin{subfigure}[b]{0.49\textwidth}
    \centering 
    \includegraphics[width=\textwidth,height=.8\textwidth]{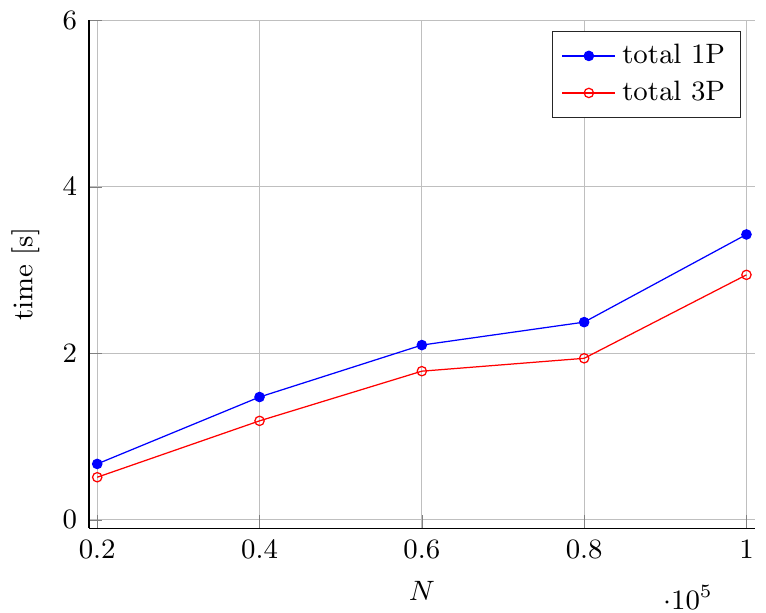}
  \end{subfigure}
  \caption{Runtime comparison of 1d- and 3d-periodic $k$-space sum vs. number of particles to achieve a relative rms error below $10^{-9}$. (Left) FFT/IFFT and scaling runtime. (Right) total runtime. We used $\xi=4$. The parameters used in this example are listed in table \ref{tab:params}.}
  \label{fig:1p3p_tole10}
\end{figure}
\begin{figure}[htbp]
  \begin{subfigure}[b]{0.49\textwidth}
    \centering 
    \includegraphics[width=\textwidth,height=.8\textwidth]{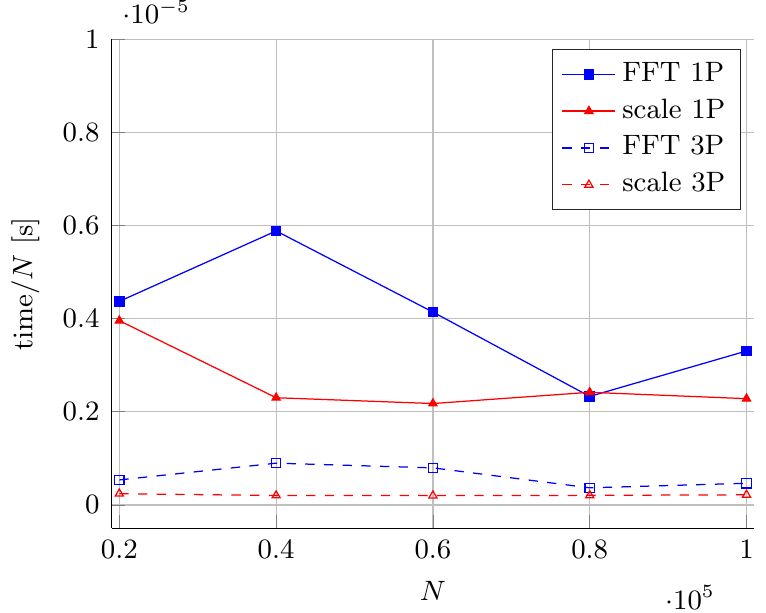}
\end{subfigure}
  \begin{subfigure}[b]{0.49\textwidth}
    \centering 
    \includegraphics[width=\textwidth,height=.8\textwidth]{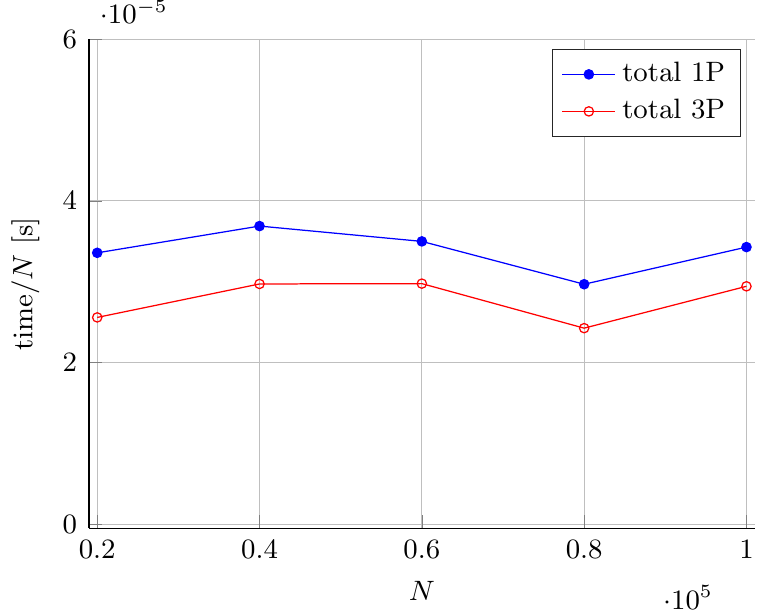}
\end{subfigure}
\caption{Per-particle runtime comparison of 1d- and 3d-periodic $k$-space sum vs. number of particles to achieve a relative rms error below $10^{-9}$. (Left) FFT/IFFT and scaling runtime. (Right) total runtime. We used $\xi=4$. The parameters used in this example are listed in table \ref{tab:params}.}
\label{fig:1p3p_tole10_scale}
\end{figure}
\end{example}

%===================================================================================================================
\begin{example}
\label{ex:wall}
In this example, we present a runtime comparison of the SE1P and SE3P method using a cloud-wall system (see figure \ref{fig:cloud_wall}) introduced in \cite{Arnold2013}. This system is constructed artificially to create a strong long-range interaction. First we consider a system of $N=300$ oppositely charged particles located in a box of size $L=10$. Then the system is scaled up such that the particle density is constant $\frac{N}{L^3}=0.3$. We compute the absolute rms error in the evaluation of the electrostatic force using both methods and compare it with the results obtained in \citep{Nestler2015} (figure 5.5). The parameters $N$, $L$, $M$ and $\xi$ are the same as in \citep{Nestler2015} and are listed in table \ref{tab:cloud_wall}. The error committed in both methods is $\approx2\times10^{-5}$. In figure \ref{fig:compare_1p3p_nestler} (left) we compare the runtime of the Fourier space sum with 1d- and 3d- periodicity when the system grows. The figure again confirms the effectiveness of the SE1P method. In figure \ref{fig:compare_1p3p_nestler} (right) we plot the relative runtime of the 1d-periodic and 3d-periodic cases using the Spectral Ewald method together with the relative runtime of the NFFT-based method reported in \cite{Nestler2015}. We emphasize here that since the results in \cite{Nestler2015} is obtained using a different computer than ours, direct runtime comparison is not feasible. However, it is still reasonable to compare the cost for 1d-periodic and 3d-periodic systems for each method separately. The SE1P and SE3P methods only differ in the scaling and the FFT/IFFT steps and since the oversampling is only applied on $20\%$ of the grid, the FFT/IFFT cost is almost the same for both methods. Moreover, the memory requirement is significantly smaller in the SE1P method compared to the 1d-periodic NFFT-based method. This is evident from figure \ref{fig:compare_1p3p_nestler} (right) in which a huge increase occurs in the runtime of the NFFT-based method for $N=1\,228\,800$.

\begin{figure}[htbp]
\centering \includegraphics[width=.35\linewidth]{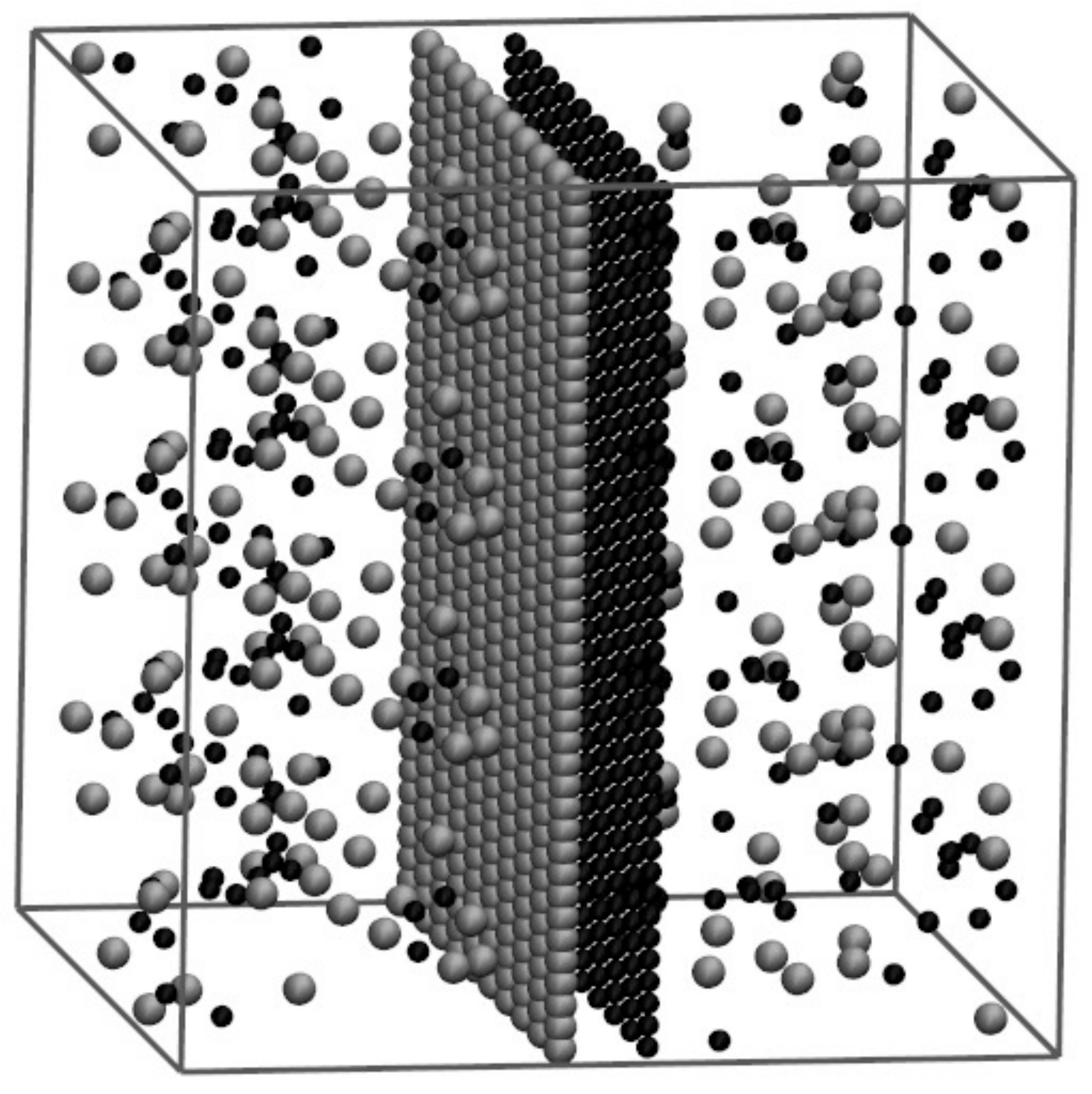}
\caption{A cloud-wall system of particles (borrowed from \cite{Arnold2013}) used in example \ref{ex:wall}.}
\label{fig:cloud_wall}
\end{figure}
\begin{figure}[htbp]
  \begin{subfigure}[b]{0.5\textwidth}
    \centering  \includegraphics[width=\textwidth,height=.8\textwidth]{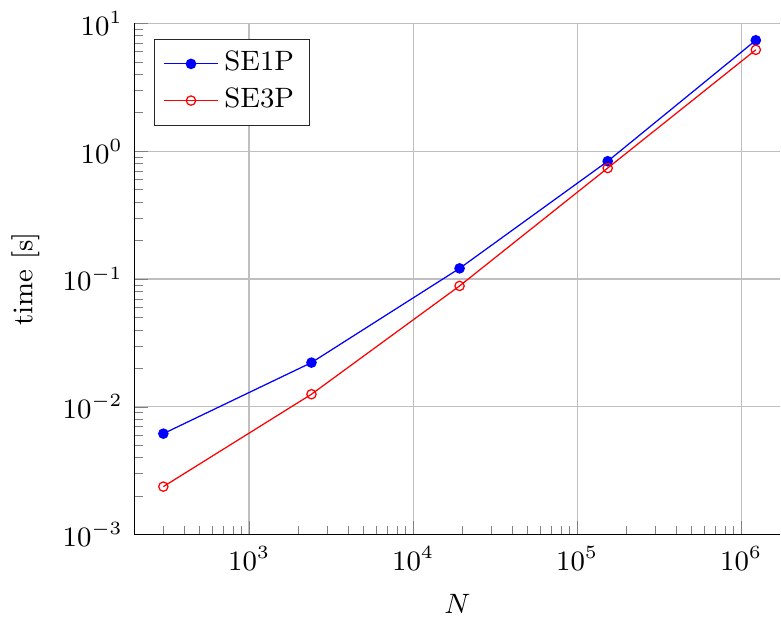}
  \end{subfigure}
  \begin{subfigure}[b]{0.49\textwidth}
    \centering  \includegraphics[width=\textwidth,height=.8\textwidth]{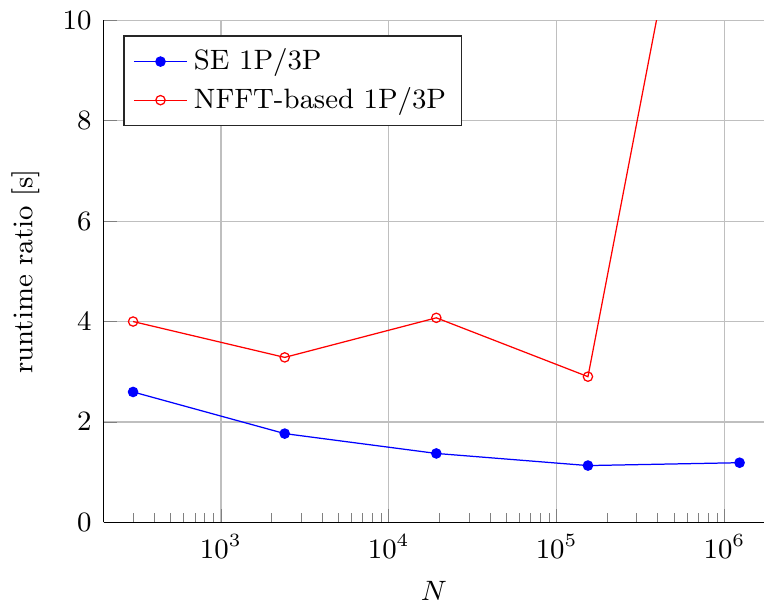}
  \end{subfigure}
  \caption{(Left) Runtime comparison of the 1d- and 3d-periodic $k$-space sums vs. number of particles to achieve rms error of $\approx 2\times 10^{-5}$. (Right) Comparison of the runtime ratio of the 1d- and 3d-periodic $k$-space sums using Spectral Ewald method and the NFFT-based method \cite{Nestler2015}. The parameters used in this example are listed in table \ref{tab:cloud_wall}.}
  \label{fig:compare_1p3p_nestler}
\end{figure}

\begin{table}[!ht]
\small
\begin{center}
\begin{tabular}{rrrccc}
\hline
$N$ & $L$ & $M$ & $\nl $ & $\sl $ & rms error \\ 
\hline
300         &  10 & 16  & 1  & 2 & 2.274e-05 \\ 
2\,400      &  20 & 32  & 2  & 2 & 1.963e-05 \\ 
19\,200     &  40 & 64  & 4  & 2 & 1.717e-05 \\ 
153\,600    &  80 & 128 & 8  & 2 & 1.561e-05 \\ 
1\,228\,800 & 160 & 256 & 16 & 2 & 1.479e-05 \\
\end{tabular}
\end{center}
\caption{List of the parameters used in example \ref{ex:wall}. The rms error in the SE1P and SE3P methods is $\approx2\times10^{-5}$ and $P=12$, $\xi\approx0.7186$ and $s_0\approx 2.4$.}
\label{tab:cloud_wall}
\end{table}
\end{example}

%===================================================================================================================
%===================================================================================================================

%===================================================================================================================
%===================================================================================================================
\section{Summary and conclusions}
We develop a fast and accurate algorithm to compute the electrostatic potential, force and energy for three dimensional systems of charged particles under singly periodic boundary conditions. The method is based on the Ewald summation formula and follows the general framework of other Particle mesh Ewald (PME) methods with an FFT treatment of the Fourier sum. Specifically, this work is an extension of the Spectral Ewald method that has been developed for triply periodic (SE3P, \cite{Lindbo2011}) and doubly periodic (SE2P, \cite{Lindbo2012}) boundary conditions.
By using suitably scaled Gaussians, approximation errors can be decoupled from truncation errors, and the needed size of the FFT grid in any periodic direction is determined by the actual Ewald sum. Controlling the approximation errors this way, the method is spectrally accurate, meaning that errors decay exponentially as the grid size increases.

FFT based methods, like the Spectral Ewald method, are most efficient for the triply periodic case where FFTs can be used in all directions without oversampling. To resolve the Fourier integrals that appear in any non-periodic direction, an upsampling of the FFT grid in the non-periodic direction is needed, and in the SE2P method it was done for all modes in the periodic directions. A plain upsampling like this of two non-periodic directions for the singly periodic problem would yield a substantial extra cost as compared to the triply periodic problem.  However, in this paper we have shown that it is sufficient to apply the upsampling in the non-periodic directions on about $20\%$ of the Fourier modes in the periodic direction. We establish an adaptive Fourier transform to apply the upsampling on a small set of Fourier modes, and with this, the cost of approximating the Fourier integrals reduces significantly. The same level of accuracy can still be achieved, provided that the local upsampling factor and the number of upsampled modes are chosen properly.

There is a term in the singly periodic Ewald sum associated with the zero wave number in the periodic direction, which evaluated directly would yield a cost of $\order{N^2}$.
To treat this zero mode term, we integrate the method proposed by Vico \etal \cite{Vico2016} for solving the free space Poisson’s equation into our framework. This zero wave number contribution then only requires a special scaling in Fourier space, and can be integrated into our adaptive Fourier transform at a negligible cost.

A typical increase in cost of the FFTs performed in the 1d-periodic as compared to the 3d-periodic case is a factor of 2-3. The gridding cost (evaluating Gaussians in the gridding and gathering steps of the algorithm) is essentially the same in both cases.
The ratio of the total runtime cost for the SE1P method introduced here and the SE3P method is therefore even smaller. From published results, we can obtain the cost ratio of computing the 1d-periodic to 3d-periodic potential by the NFFT-based method proposed by Nestler \etal \cite{Nestler2015}, and we can note that the cost ratio for our method is smaller. Furthermore, our algorithm is relatively simpler for implementation and is more efficient in terms of memory requirement.

The method proposed here can be extended to other applications. The spectral Ewald method has already been developed for the fast summation of periodic \cite{Lindbo2011,Lindbo2012} as well as free space Stokes potentials \cite{Klinteberg2016}, and the singly periodic case would be a natural extension.

%===================================================================================================================
%===================================================================================================================

%===================================================================================================================
%===================================================================================================================
\section*{Acknowledgement}
This work has been supported by the the Swedish Research Council under grant no. 2011-3178 and by the Swedish e-Science Research Center. The authors gratefully acknowledge this support.

%===================================================================================================================
%===================================================================================================================

%===================================================================================================================
%===================================================================================================================
\begin{appendix}
\section{Direct sum evaluation}
\label{subsec:Incomp_bessel}
In this work, the direct sum will be used as a reference solution for the fast method. We will here comment on how to accurately evaluate the modified Bessel function of the second kind, ${\bf K_0}$, in \eqref{eq:ewald1p_fourier_complete} and the exponential integral in \eqref{eq:ewald1p_zero}.

The computation of ${\bf K_0}$ have been the subject of many articles and different approaches have been proposed, e.g.,\ \cite{Harris2009a,Slevinsky2010}. Many numerical subtleties arise while evaluating the function, specifically when the arguments are extremely small. Here we present an accurate approach to evaluate this function, which however is not the most efficient way to compute the modified Bessel function. The reader may consult \cite{Harris2009a} and \cite{Slevinsky2010} for more details. This approach is based on splitting the integration interval and applying Gauss-Legendre quadrature on each interval. If $a>b$, we set $v=\min(\sqrt{b},1)$ and compute the integral on $[0,v]$ and $[v,1]$. If $a<b$, we use equality
$${\bf K_0}(a,b)=2k_0(2\sqrt{ab})-{\bf K_0}(b,a),$$
where $k_0(\cdot)$ is the modified Bessel function of the second kind and is available e.g\ as \textsf{besselk} in MATLAB and GNU Scientific Library (GSL) \cite{gsl}. We set $v=\min(b\sqrt{\frac{b}{a}},1)$ and compute the integrals. With this approach, the absolute error is $\approx\pm10^{-15}$.

Moreover we need to compute the direct sum in \eqref{eq:ewald1p_zero}. We illustrate how to evaluate the exponential integral $\E(\cdot)$ accurately.

The general exponential integral is defined as \cite[Sec 6.3]{Press1987},
\begin{align}
\text{E}_n(x)=\int_1^{\infty}\dfrac{e^{-xt}}{t^n}\dif t, \quad x>0,\quad n\in\mathbb{N}\cup\lbrace0\rbrace.
\end{align}
This function can be represented with a continued fraction for $x\gtrsim 1$ as
\begin{align}
\text{E}_n(x) = e^{-x}\left( \dfrac{1}{x+n-}~\dfrac{1\cdot n}{x+n+2-}~\dfrac{2(n+1)}{x+n+4-}\cdots\right), \quad x\gtrsim 1,
\label{eq:expint_n_fraction}
\end{align}
and a series representation for $0<x<1$ as
\begin{align}
\text{E}_n(x)=\dfrac{~~~(-x)^{n-1}}{(n-1)!}[-\log(x)+\psi(n)]-\sum^{\infty}_{\subindex{m=0}{m\neq n-1}}\dfrac{(-x)^m}{(m-n+1)!},\quad 0<x<1,
\label{eq:expint_n_series_rep}
\end{align}
where 
\begin{align*}
\psi(1)=-\gamma,\quad \psi(n) = -\gamma+\sum_{m=1}^{n-1}\dfrac{1}{m}.
\end{align*}
For the case $n=1$, equation \eqref{eq:expint_n_series_rep} can be written as
\begin{align}
\E(x) = -\gamma-\log(x)-\sum_{m=1}^ {\infty}(-1)^m\dfrac{x^m}{m!m}.
\label{eq:expint_1_series_rep}
\end{align}
It follows from equation \eqref{eq:expint_1_series_rep} that
\begin{align}
\lim_{x\to0}\left\lbrace \gamma+\log(x)+\E(x)\right\rbrace=0.
\label{eq:limit}
\end{align}
Due to the charge neutrality condition $\sum_{\ni=1}^Nq_{\ni}\gamma=0$, but the present form of $\varphi^{\F,k_3=0}$ in equation \eqref{eq:ewald1p_zero} with extra $\gamma$ is more convenient to compute. Also equation \eqref{eq:limit} shows that \eqref{eq:k0_term} is valid for small values of $\rho_{\mi\ni}^2\xi^2$ and the expression $\gamma+\log(\cdot)+\E(\cdot)$ can be computed via the truncated version of the series in \eqref{eq:expint_1_series_rep}. Another approach is to compute the exponential function using Chebyshev interpolation. For this, the Chebyshev coefficients are evaluated and stored for different sets of input arguments. These coefficients can then be used to evaluate the exponential integral. The computation of the exponential integral can be accelerated by considering the fact that for $x>34$, $\E(x)$ is less than machine precision. A fast implementation of this algorithm is available in GSL Library.

%===================================================================================================================
%===================================================================================================================

%===================================================================================================================
%===================================================================================================================
\section{Proof of Theorem \ref{theo:spectral_accuracy}}
\label{appendix:proof}
Before presenting the proof of the theorem \ref{theo:spectral_accuracy}, we first review the method of contour integrals proposed by Donaldson and Elliot \cite{Donaldson1972} to derive an accurate estimate for the trapezoidal quadrature error. The reader may also consult the valuable survey by Trefethen and Weideman \cite{Trefethen2014}. Since in theorem \ref{theo:spectral_accuracy}, the integrand is defined on the real line, we will focus on integrals of the type
\begin{align}
I = \int_{\mathbb{R}} f(x)\dif x.
\end{align}
Consider the following two elementary definitions.
\begin{definition}
The \textit{residue} of a function $f$ at an order $m$ pole, $z_0$, is denoted by $\res[f,z_0]$ and defined as
\begin{align*}
\res[f,z_0] = \dfrac{1}{(m-1)!}\lim_{z\to z_0} \dfrac{\text{d}^{m-1}}{\dif z^{m-1}}(z-z_0)^mf(z).
\end{align*}
\end{definition}
\begin{definition}
A \textit{meromorphic} function is a single-valued function that is analytic everywhere except at a finite set of poles.
\end{definition}
Let $f(x)$ be a smooth function which decays at infinity and define a $(2n+1)$-point
trapezoidal rule by
\begin{align*}
I_{n,h} = h\sum_{j=-n}^{n} f(jh),\quad h>0.
\end{align*}
We define the remainder function as 
\begin{align*}
R_{n,h} := I-I_{n,h}.
\end{align*}
Now assume that $\C$ is a contour enclosing the interval $[-nh,nh]$ on which $f$ is analytic. The main principle of using contour integrals to estimate the quadrature error is as follows: Consider a meromorphic function $\psi_h(z)$ with simple poles at quadrature points $z_j=jh$, $j\in\mathbb{Z}$ for which $\res[\psi_h(z),z_j]=\frac{h}{2\pi \ii}$.
Therefore
\begin{align}
I_{n,h}= h\sum_{j=-n}^{n} f(jh) = \int_\C f(z)\psi_h(z)\dif z.
\label{eq:contour_In}
\end{align}
Moreover, there exist a meromorphic function $\phi(z)$ such that 
\begin{align}
I = \int_\C f(z)\phi(z)\dif z.
\label{eq:contour_I}
\end{align}
Hence the remainder of the trapezoidal rule can be written as
\begin{align}
R_{n,h} = \int_\C f(z)m(z)\dif z = \int_\C f(z)(\phi(z)-\psi_h(z))\dif z.
\label{eq:rnh}
\end{align}
In fact, this technique can be used for any quadrature rule by defining an appropriate characteristic function $\psi_h(z)$ that has simple poles at quadrature nodes inside the contour and residues equal to the quadrature weights divided by $2\pi \ii$. It can be shown that for a function $f$ defined on the real line, 
\begin{align}
\phi(z) = \left\lbrace
\begin{array}{rc}
\frac{1}{2}, & \im(z)<0, \\
\vspace{-.2cm}\\
-\frac{1}{2}, & \im(z)>0,
\end{array}
\right.
\label{eq:charact_phi}
\end{align}
and 
\begin{align}
\psi_h(z) = -\dfrac{\ii}{2}\cot\left(\dfrac{\pi z}{h}\right),
\label{eq:charact_psi}
\end{align}
satisfy \eqref{eq:contour_In} and \eqref{eq:contour_I} .
Hence, $m(z)$ in \eqref{eq:rnh} can be written as
\begin{align}
m(z) &= \left\lbrace
\begin{array}{rc}
\frac{1}{2}(1+\ii\cot(t)), & \im(z)<0, \\
\vspace{-.2cm}\\
-\frac{1}{2}(1-\ii\cot(t)), & \im(z)>0,
\end{array}
\right. \\
&=\left\lbrace
\begin{array}{rc}
\dfrac{1}{1-e^{2\ii t}}, & \im(z)<0, \\
\vspace{-.2cm}\\
\dfrac{-1}{1-e^{-2\ii t}}, & \im(z)>0,
\end{array}
\right.
\label{eq:charct_diff}
\end{align}
where $t=\dfrac{\pi z}{h}$.

If $f$ has poles at $z_j$, the contour $\C$ can be modified such that each pole is enclosed by a circle. Now in the limit when the radius of circles go to zero, we have,
\begin{align*}
R_{n,h}  = \int_\C f(z)m(z)\dif z-2\pi \ii\sum_{z_j}\res[f(z)m(z),z_j].
\end{align*}
Now we are ready to present the proof of Theorem \ref{theo:spectral_accuracy}.
\begin{proof}{(\it Proof of Theorem \ref{theo:spectral_accuracy})}
Suppose that the integral
\begin{align}
I=\int_{\mathbb{R}} \dfrac{e^{-\alpha(k^2+k_3^2)}}{k^2+k_3^2}\dif k,\quad k_3>0,
\label{eq:proof_int}
\end{align}
is approximated using a $(2n+1)$-point trapezoidal quadrature 
\begin{align*}
I_{n,h} = h\sum_{j=-n}^nf(jh),
\end{align*}
and define the remainder $R_{n,h}:=I-I_{n,h}$. Also in the limit when $n\to\infty$, we define $(I_h,R_h) = \lim_{n\to\infty} (I_{n,h},R_{n,h})$. We aim to derive an accurate error estimate for $R_h$ using the residue calculus. Consider the integrand in \eqref{eq:proof_int} in the complex plane,
\begin{align*}
f:\mathbb{C}\to\mathbb{C}, \quad f(z) = \dfrac{e^{-\alpha(z^2+k_3^2)}}{z^2+k_3^2},
\end{align*}
Clearly, $f$ has two simple poles at $z_j=\pm \ii k_3$. An appropriate choice of the contour $\C$ is a positively oriented rectangle that encloses the interval $[-nh,nh]$ and the poles $z_j$. More specifically consider a contour that passes through $(n+\frac{1}{2})h\pm \ii a$ and $-(n+\frac{1}{2})h\pm \ii a$ points, where $a>k_3$, see figure \ref{fig:contour}. 
\begin{figure}[htbp]
  \centering \includegraphics[width=0.49\linewidth]{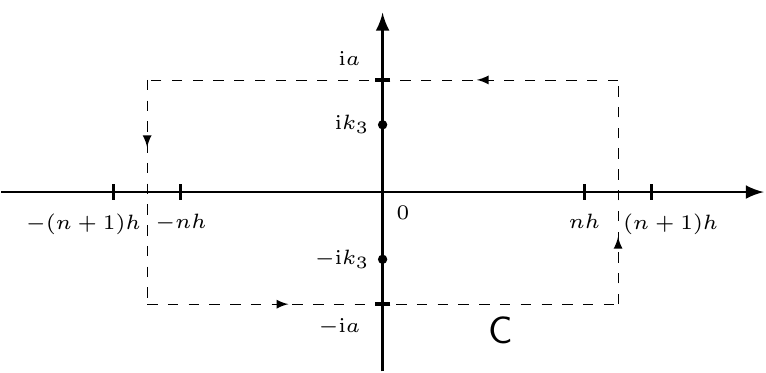}
  \caption{}
  \label{fig:contour}
\end{figure}
Since $f(z)m(z)$ decays at infinity, therefore $\int_{\C}\to0$ as $\C$ goes to infinity. Hence $R_{h}$ can be determined solely by the residues as
\begin{align*}
R_h  = -2\pi \ii\sum_{z_j}\res[f(z)m(z),z_j],
\end{align*}
where $m(z)$ is defined as in \eqref{eq:charct_diff}. Therefore,
\begin{align*}
|R_h|  = \dfrac{2\pi}{k_3}\dfrac{1}{e^{2\pi k_3/h}-1},
\end{align*}
\end{proof}
\end{appendix}

\bibliographystyle{abbrvnat_mod}
\bibliography{library}
\end{document}